\documentclass[11pt]{article}

\usepackage[a4paper,margin=1in]{geometry}
\usepackage{amsmath}

\usepackage{graphicx}
\usepackage{amssymb}
\usepackage{amsmath}

\usepackage{verbatim,booktabs}
\usepackage{subfigure,epsfig,url,psfrag}
\usepackage{float}
\usepackage{mathrsfs}

 \usepackage[usenames,dvipsnames]{pstricks}
 \usepackage{epsfig}
 \usepackage{pst-grad} % For gradients
 \usepackage{pst-plot} % For axes
  \usepackage{enumitem}

\usepackage[ruled,vlined]{algorithm2e}

\makeatletter
\newcommand{\setwindow}[5]{
\def\xmin{#1}%
\def\ymin{#2}%
\def\xmax{#3}%
\def\ymax{#4}%
\pstFPsub\viewingwidth{#3}{#1}%
\pstFPdiv\result{\strip@pt#5}{\viewingwidth}%
\psset{unit=\result pt}}
\makeatother

\usepackage{amsthm}
\newtheorem{theorem}{Theorem}
\newtheorem{corollary}{Corollary}
\newtheorem{proposition}{Proposition}
\newtheorem{lemma}{Lemma}

\newtheorem{remark}{Remark}

\def\st{{\rm s.t.}}

\def\R{{\mathbb R}}
\def\Ri{{\mathcal R}}
\def\Sn{{\mathbb S}}
\def\cl{{\rm cl}}
\def\conv{{\rm conv}}

\def\ri{{\rm ri}}
	
\def\ie{{i.e.,} }

\def\G{{\mathcal G}}
\def\S{{\mathcal S}}
\def\H{{\mathcal H}}
\def\C{{\mathcal C}}

\def\I{{\mathcal I}}

\def\B{{\mathcal B}}

\def\I{{\mathcal I}}

\def\D{{\mathcal D}}
\def\U{{\mathcal U}}

\def\P{{\mathcal P}}
\def\Q{{\mathcal Q}}
\def\A{{\mathcal A}}
	
\def\F{{\mathcal F}}

\def\01{\ensuremath{0\mathord{-}1}}

\def\E{{\mathcal E}}

\def\Z{{\mathcal Z}}

\allowdisplaybreaks

\newcommand{\bi}{\begin{list}{$\bullet$}{\setlength{\parsep}{0pt}\setlength{\itemsep}{0pt}}}

\def\XXsum#1#2#3{{\setbox0=\hbox{$#1{#2#3}{\sum}$ }
\vcenter{\hbox{$#2#3$ }}\kern-.5\wd0}}

\newcounter{claim} %[section]
\def\XXsum#1#2#3{{\setbox0=\hbox{$#1{#2#3}{\sum}$ }
\vcenter{\hbox{$#2#3$ }}\kern-.5\wd0}}

\usepackage[ruled,vlined]{algorithm2e}

\title{Explicit convex hull description of bivariate quadratic sets with indicator variables
\thanks{Antonio De Rosa was supported in part by the NSF DMS Grant No.~1906451, the NSF DMS Grant No.~2112311, and the NSF DMS CAREER Award No.~2143124.}}
\author{Antonio De Rosa
\thanks{Department of Mathematics, University of Maryland, 4176 Campus Dr, College Park, MD 20742, USA.
             E-mail: {\tt  anderosa@umd.edu}.
             }
\and
Aida Khajavirad
\thanks{Department of Industrial and Systems Engineering, Lehigh University, Bethlehem, PA 18015, USA.
             E-mail: {\tt aida@lehigh.edu}.
             }
}

\begin{document}

\maketitle

\date{}
\begin{abstract}
We consider the nonconvex set
$\S_n = \{(x,X,z): X = x x^T, \; x (1-z) =0,\; x \geq 0,\; z \in \{0,1\}^n\}$,
which is closely related to the feasible region of several difficult nonconvex optimization problems such as the best subset selection and constrained portfolio optimization. Utilizing ideas from convex analysis and disjunctive programming,
we obtain an explicit description for the closure of the convex hull of $\S_2$ in the space of original variables.
In order to generate valid inequalities corresponding to supporting hyperplanes of the convex hull of $\S_2$,
we present a simple separation algorithm that can be incorporated in branch-and-cut based solvers to enhance the quality of existing relaxations.
\end{abstract}


\vspace{0.1in}


\section{Introduction}
\label{sec:intro}

Consider a mixed-integer quadratically constrained optimization problem (MIQCP) with indicator variables of the form:
\begin{align}
\label{miqp}
\tag{QPI}
\min \quad & x^T Q_0 x + c_0^T x\\
\st  \quad & x_i (1-z_i) = 0, \quad \forall i \in \{1,\ldots,n\}\nonumber\\
& x^T Q_j x + c_j^T x \leq b_j, \quad \forall j \in \{1,\ldots,m\}\nonumber\\
& A x \leq \alpha, \; B z \leq \beta\nonumber\\
&x \in \R^n, \; z \in \{0,1\}^n,\nonumber
\end{align}
where $A, B, Q_j$, $j \in \{0,\ldots,m\}$, are matrices of appropriate dimensions and $Q_j$ is symmetric for all $j \in \{0,\ldots,m\}$.
The binary variables $z$ are often referred to as~\emph{indicator variables}.
Problem~\eqref{miqp} subsumes several classes of difficult nonconvex optimization problems such as best subset selection~\cite{bertsimas16},
constrained portfolio optimization~\cite{bien96}, quadratic facility location~\cite{gunlin10}, and optimal power flow with transmission switching~\cite{hijazi17,best20}, among others.
It is important to note that we are not assuming that any of the matrices $Q_j$
is positive semi-definite (PSD). In this paper, we are interested in the quality of convex relaxations for Problem~\eqref{miqp}.

To build convex relaxations for MIQCPs, it is common practice to first introduce auxiliary variables $X_{ij} := x_i x_j$ for all $i,j \in [n]:= \{1,\ldots,n\}$ to obtain a reformulation of Problem~\eqref{miqp} in a lifted space:
\begin{align}
\label{miqp2}
\tag{QPI$'$}
\min \quad & \langle Q_0,X \rangle + c_0^T x\\
\st  \quad  &  x_i (1-z_i) = 0, \quad \forall i \in [n] \nonumber\\
& X = xx^T\nonumber\\
& \langle Q_j,X \rangle + c_j^T x \leq b_j, \quad \forall j \in [m]\nonumber\\
& A x \leq \alpha, \; B z \leq \beta\nonumber\\
&x \in \R^n, \; z \in \{0,1\}^n, \; X \in \Sn^n_{+},\nonumber
\end{align}
where $\langle \cdot,\cdot \rangle$ denotes the standard matrix inner product, and $\Sn^n_{+}$ denotes the set of $n \times n$ PSD matrices.
Throughout this paper, we assume that $A x \leq \alpha$ implies $x \in \R_{+}^n$, where $\R^n_{+}$ denotes the set of nonnegative real $n$-vectors.
It then follows that to construct strong convex relaxations for Problem~\eqref{miqp2}, it is essential to effectively convexify the nonconvex set:
\begin{equation}\label{TheSet}
\S_n = \Big\{(x,X,z):\; X = x x^T, \; x (1-z) =0,\; x \in \R_{+}^n, \; z \in \{0,1\}^n\Big\}.
\end{equation}
Throughout this paper, given a set $\C$, we denote its convex hull by $\conv(\C)$, and the closure of its convex hull by $\overline{\conv}(\C)$.
The simplest case of $\S_n$, \ie the case with $n=1$, has been studied in~\cite{akturk09,gunlin10}; leveraging on disjunctive programming techniques~\cite{rf70,b85,cs99}, the authors obtain an explicit characterization of the closure of the convex hull of $\S_1$ in the space of the original variables:
\begin{proposition}[\cite{akturk09,gunlin10}]\label{n=1}
The closure of the convex hull of $\S_1$ is given by
\begin{equation}\label{n1c}
\overline{\conv}(\S_1) = \Big\{(x,X,z): \; X_{11} z_1 \geq x^2_1, \; x_1 \geq 0,\; X_{11} \geq 0, \; z_1 \in [0,1]\Big\}.
\end{equation}
%where we assume that if $z_1 = x_1 = 0$, then $\frac{x^2_1}{z_1} = 0$ and if $z_1 = 0$, $x_1 > 0$, then $\frac{x^2_1}{z_1} = +\infty$.
\end{proposition}
Using Proposition~\ref{n=1}, one obtains the following convex relaxation of Problem~\eqref{miqp2}:
\begin{align}
\label{perspR}
\tag{Persp}
\min \quad & \langle Q_0,X \rangle + c_0^T x\\
\st  \quad  & X \succeq xx^T\nonumber\\
& X_{ii} z_i \geq x^2_i, \quad \forall i \in [n]\nonumber\\
& \langle Q_j,X \rangle + c_j^T x \leq b_j, \quad \forall j \in [m]\nonumber\\
& A x \leq \alpha, \; B z \leq \beta\nonumber\\
&x \in \R^n, \; z \in [0,1]^n, \; X \in \Sn^n_{+}.\nonumber
\end{align}
Problem~\eqref{perspR} is often referred to as the~\emph{perspective relaxation}; this relaxation has been shown to be highly effective
in the special case where $m = 0$, and $Q_0$ is a PSD matrix that is diagonal or is strongly diagonally dominant~\cite{fg:06,gunlin10,dong15}.

To further improve the quality of convex relaxations for Problem~\eqref{miqp2}, it is natural to study the facial structure of the convex hull of $\S_2$ defined by:
\begin{align}\label{DudeSet}
\S_2 := \Big\{& (x,X,z): X_{11}=x^2_1, \; X_{12}=x_1 x_2, \; X_{22}=x^2_2, \; x_1 (1-z_1) =0, \; x_2(1-z_2) = 0, \nonumber\\
            & x_1, x_2 \geq 0, \; z_1, z_2 \in \{0,1\}\Big\}.
\end{align}
In~\cite{ansbur21}, the authors consider the set $\S_2$ with the additional constraints $x_1, x_2 \leq 1$, and obtain an extended formulation, containing three auxiliary variables, for the convex hull of this set.  Their representation consists of second-order cone constraints and PSD conditions, and hence can be embedded in conic solvers. To this date, obtaining an explicit characterization for the convex hull of $\S_2$ in the original space $(x,X,z)$ remains an open question. The main obstruction in obtaining such an explicit description is that there exists no general methodology to perform projection when nonlinear expressions are involved. This is in contrast with the polyhedral setting in which standard techniques such as Fourier-Motzkin elimination can be applied to obtain the projection in the original space. One of the goals of this paper is to build necessary tools to perform projections of convex sets defined by nonlinear inequalities.

%\vspace{0.1in}
\paragraph{Convex quadratic optimization with indicator variables.}
In the special case where $m = 0$ and $Q_0 \in \Sn^n_{+}$, Problem~\eqref{miqp2} simplifies to
a convex quadratic optimization problem with indicator variables. Due to its important applications in signal processing, portfolio optimization, and machine learning
(see~\cite{Lodi21} for an extensive survey), this problem has been widely studied in the literature~\cite{dong15,jeon17,atagomhan21,hangomata20,ansbur21,fraGen20,weiet22,wei22}.
In~\cite{atagomhan21,hangomata20}, the authors study the convex hull of the epigraph of bivariate convex quadratic functions. Namely, in~\cite{atagomhan21}, the authors consider the nonconvex set
$$\Z^{-}_2:=\Big\{(x,t,z): t \geq d_1 x^2_1-2 x_1 x_2 + d_2 x^2_2, \; x_i(1-z_i)=0, x_i\geq 0, z_i\in \{0,1\}, i \in \{1,2\}\Big\},$$
with $d_1,d_2 > 0$, $d_1 d_2 \geq 1$, and derive $\overline{\conv}(\Z^{-}_2)$ in the original space. Subsequently, in~\cite{hangomata20},
they consider
$$\Z^{+}_2:=\Big\{(x,t,z): t \geq d_1 x^2_1+2 x_1 x_2 + d_2 x^2_2, \; x_i(1-z_i)=0, x_i\geq 0, z_i\in \{0,1\}, i \in \{1,2\}\Big\},$$
with $d_1,d_2 > 0$, $d_1 d_2 \geq 1$ and derive $\overline{\conv}(\Z^{+}_2)$ in original space. Furthermore, utilizing an extended formulation for a relaxation of the convex hull of $\S_2$, they show that the resulting semidefinite programming relaxation outperforms Problem~\eqref{perspR} in terms of relaxation quality but is more expensive to solve due to the added variables associated with the extended formulation.
In~\cite{weiet22}, the authors consider the special case where $Q_0$ is a rank-one matrix, and obtain an explicit description for the convex hull of the set
$$
\Z_{\H}=\Big\{(x,z,t): (q^T x)^2 \leq t, \; x(1-z)=0, \; x \in \R^n, \; z \in \H\Big\},
$$
where $q \in \R^n$ such that $q_i \neq 0$ for all $i \in [n]$, and $\H \subseteq \{0,1\}^n$.

\paragraph{Our contributions.}
In this paper, we obtain an explicit  characterization for the closure of the convex hull of $\S_2$ as defined by~\eqref{DudeSet} in the space of original variables.
Using disjunctive programming, we first obtain a convex extended formulation for $\overline{\conv}(\S_2)$.
We then project out the auxiliary variables to express $\overline{\conv}(\S_2)$ in the original space;
namely, by building on ideas from convex analysis, we show that this projection can be performed via the analytical solution of a parametric
convex optimization problem.  This approach in turn enables us to obtain a piece-wise characterization of the convex hull, in which different
pieces are obtained by intersecting the convex hull with the elements of a partition of the search space. For each of these pieces, while the defining inequalities associated to the boundary of the convex hull are not globally valid, they locally
define the boundary of the convex hull. Hence the corresponding first-order Taylor expansions are supporting hyperplanes of the convex hull and therefore are globally valid.
Indeed, we show that the separation problem over  $\overline{\conv}(\S_2)$ can be solved via a simple algorithm.
This algorithm can readily be incorporated in branch-and-cut based global solvers~\cite{IdaNick18,antigone14,BestuzhevaEtal2021ZR} to improve the quality of existing relaxations for nonconvex
problems containing quadratic sets with indicator variables. Similar cutting plane algorithms based on convexification results of~\cite{ts02,ks:11mp} have been already successfully implemented in the global solver {\tt BARON}~\cite{IdaNick18}.

The rest of this paper is organized as follows. In Section~\ref{state} we provide statements of our main results.
In Section~\ref{separate} we present our separation algorithm.
Section~\ref{proof1} contains the proof of our convex hull characterization. Finally, in Section~\ref{lemmata} we present all technical lemmata that we need to prove our main results.

\section{Main results}
\label{state}

The main purpose of this paper is to characterize the closure of the convex hull of $\S_2$ in the space of original variables.
To present this characterization, we first introduce some notation. Given a set $\C$, we denote its relative interior by ${\rm ri}(\C)$, its closure by $\cl(\C)$,
its boundary by $\partial \C$, and its complement by $\C^c$. We denote by $\B_r(x)$ the open ball of radius $r$ centered at $x$ in the Euclidean norm.
Consider the function $f(u,v) := \frac{u^2}{v}$, $u \in \R$ and $v > 0$. We define the closure of $f$,
denoted by $\hat f$, as follows:
\begin{align}\label{closure}
\hat f(u,v) :=
\begin{cases}
 \frac{u^2}{v}, \qquad &{\rm if}\; v > 0\\
 0,           \; \;  \qquad     &{\rm if} \; u=v=0\\
 +\infty        \qquad      &{\rm if} \; u\neq 0, \; v=0.
\end{cases}
\end{align}
%For notational simplicity, in the following when we write a function of the form $\frac{u^2}{v}$, we imply its closure as defined above.
Moreover, consider the function $f(u,v,w) := \frac{u v}{w}$, $u, v \geq 0$ and $w > 0$. We define the closure of $f$, denoted by $\hat f$, as follows:
\begin{align}\label{closure2}
\hat f(u,v, w) :=
\begin{cases}
 \frac{u v}{w}, \qquad &{\rm if}\; w > 0\\
 0,           \; \;  \qquad    & {\rm if} \; u v=w=0\\
 +\infty        \qquad      & {\rm if} \; u\neq 0, \; v \neq 0, \; w = 0.
\end{cases}
\end{align}
For notational simplicity, in the following when we write a function of the form $\frac{u^2}{v}$ or $\frac{u v}{w}$, we imply its closure as defined above.

\vspace{0.1in}

The following theorem states our main result:
\begin{theorem}\label{theTh}
Consider the set $\S_2$ defined by~\eqref{DudeSet}. Then there exists a convex set  $\tilde \S$ such that
$\ri(\conv(\S_2)) \subseteq \tilde \S \subseteq \overline{\conv}(\S_2)$,
and the closure of the convex hull of $\S_2$ is given by:
$$
\overline{\rm \conv}(\S_2) =  \bigcup_{i=1}^8{\cl(\tilde \S \cap \Ri_i)},
$$
%\bigcup_{i=1}^9{\overline{\conv}(\S) \cap \cl(\Ri_i)} =
where the sets $\Ri_i$, $i \in \{1,\ldots, 8\}$ satisfy $\Ri_i \cap \Ri_j = \emptyset$ for all $i \neq j$ and $\bigcup_{i=1}^8{\Ri_i} \supseteq \overline{\conv}(\S_2)$,
and
\begin{itemize}[leftmargin=*]
% i = 1
%\item if $i=1$, then
%\begin{equation*}
%\Ri_1 := \Big\{(x, X, z): \; x_1 x_2 (z_1 + z_2-1) \leq X_{12} z_1 z_2 \leq x_1 x_2 \min\{z_1, z_2\} \Big\},
%\end{equation*}
%and
%\begin{equation*}
%\overline{\conv}(\S) \cap \Ri_1 = \Big\{(x, X, z): X_{11} \geq \frac{x_1^2}{z_1}, \; X_{22} \geq \frac{x_2^2}{z_2},\; X_{12} \geq 0, \; x_1, x_2 \geq 0, \; z_1, z_2 \in [0,1]\Big\} \cap \Ri_1.
%\end{equation*}

% i =1,2,6
\item [I.] if $i \in \{1, 2, 6\}$, then
\begin{align*}
\cl(\tilde \S \cap \Ri_i)
% &=\cl(\{(x,X,z): X_{11} \geq \frac{x^2_1}{z_1}, \;  x_1 > 0, x_2 > 0, \; X_{12} > 0, \; z_1 > 0, \; z_2 \leq 1, \; z_1 \leq z_2, \\
%& \qquad X_{12} z_2 > x_1 x_2, \;  X_{12} z_1 \leq x_1 x_2, \; x^2_1 (z_2-z_1)(X_{22} z_2 - x^2_2) \geq z_1 (X_{12} z_2-x_1 x_2)^2 \})\\
=\Big\{&(x,X,z): X_{11} \geq \frac{x^2_1}{z_1}, \; X_{22} \geq \frac{x^2_2}{z_2}, \; x_1, x_2 \geq 0, \; z_1, z_2 \in [0,1]\Big\}\cap \cl(\Ri_i),
\end{align*}
where
%\begin{equation*}
%\Ri_1 := \Big\{(x, X, z): \; x_1 x_2 (z_1 + z_2-1) \leq X_{12} z_1 z_2 \leq x_1 x_2 \min\{z_1, z_2\} \Big\},
%\end{equation*}
%
\begin{align*}
%\Ri_0 := \Big\{&(x, X, z): \; z_1 = 0, \; {\rm or} \; z_2 = 0, \; {\rm or} \; X_{12} =0 \Big\},\\
\Ri_1 := \Big\{&(x, X, z): \; x_1 x_2 (z_1 + z_2-1) \leq X_{12} z_1 z_2, \; X_{12} \max\{z_1,z_2\} \leq x_1 x_2, \; X_{12} >0, \; z_1, z_2>0\Big\}  \\
%          \Big\{&(x, X, z): \; z_2 < z_1, \; x_1 x_2 (z_1 + z_2-1) \leq X_{12} z_1 z_2, \; X_{12} z_1 \leq x_1 x_2,\; X_{12} >0, \; z_2 >0\Big\} \cup \\
       &   \cup \Big\{(x, X, z): X_{12} =0 \Big\},\\
\Ri_2:=\Big\{&(x,X,z): \; z_1 \leq z_2, \; X_{12} z_2 > x_1 x_2, \; X_{12} z_1 \leq x_1 x_2, \; X_{12} >0, \; z_1>0,\nonumber\\
                       &  x^2_1 (z_2-z_1)(X_{22} z_2 - x^2_2) \geq z_1 (X_{12} z_2-x_1 x_2)^2\Big\},\\
\Ri_6:=\Big\{&(x,X,z): \; X_{12} z_1 z_2 < x_1 x_2 (z_1 + z_2 -1), \; X_{12} >0, \; z_1, z_2 >0,\\
                        & (1-z_1)(z_1+z_2-1) x^2_1 (X_{22} z_2 - x^2_2) \geq \Big(X_{12} z_1 z_2-x_1 x_2(z_1+z_2-1)\Big)^2\Big\}.
\end{align*}

% i = 3,4
\item [II.] If $i\in \{3,4\}$, then
\begin{align*}
\cl(\tilde \S \cap \Ri_i)  = \Big\{(x,X,z): \; & X_{22} \geq \frac{x^2_2}{z_2}, \; \Big(X_{11} - \frac{x^2_1}{z_2}\Big)\Big(X_{22}- \frac{x^2_2}{z_2}\Big) \geq \Big(X_{12}- \frac{x_1 x_2}{z_2}\Big)^2, \; x_2 \geq 0, \nonumber\\
&   z_1, z_2 \leq 1 \Big\} \cap \cl(\Ri_i),
\end{align*}
where
\begin{align*}
\Ri_3:=\Big\{&(x,X,z): z_1 < z_2, X_{12} x_2 > X_{22} x_1, z_1(X_{12} z_2 - x_1 x_2)^2 > x^2_1(z_2-z_1)(X_{22} z_2 - x^2_2), \\
              & x_1 \geq 0, \; X_{12} >0, \; z_1>0\Big\},\\
\Ri_4:=\{&(x,X,z): z_2 \leq z_1, \; X_{12} x_2 > X_{22} x_1,\; x_1 \geq 0, \; X_{12} >0, \; z_2>0\}.
\end{align*}

%&\Ri_4:=\{(x,X,z): \; z_2 \leq z_1, \; X_{12} x_2 \geq X_{22} x_1 , \; X_{12} z_2 > x_1 x_2\}.

% i = 5
\item [III.] If $i=5$, then
\begin{align*}
\cl(\tilde \S \cap \Ri_i)= \Big\{(x,X,z): \; & X_{22} \geq \frac{x^2_2}{z_2}, \; \Big(X_{11} - \frac{x^2_1}{z_1}\Big)\Big(X_{22} - \frac{x^2_2}{z_1}\Big) \geq \Big(X_{12}  -\frac{x_1 x_2}{z_1}\Big)^2, \;  x_1\geq 0,\nonumber\\
&  z_1, z_2 \leq 1\Big\} \cap \cl(\Ri_5),
\end{align*}
where
$$
\Ri_5 = \{(x,X,z): \; X_{12} z_1 > x_1 x_2, \; X_{22} x_1 \geq X_{12} x_2, \; x_2 \geq 0, \; X_{12} >0, \; z_1, z_2>0\},
$$

% i = 6
%\item if $i=6$, then
%\begin{align*}
% \cl(\tilde \S \cap \Ri_6)=\Big\{&(x,X,z): \;  X_{22}\geq \frac{x^2_2}{z_1}, \; \Big(X_{11} - \frac{x^2_1}{z_1}\Big)\Big(X_{22} - \frac{x^2_2}{z_1}\Big) \geq \Big(X_{12}-\frac{x_1 x_2}{z_1}\Big)^2, x_1, x_2 \geq 0, \nonumber\\
%& X_{12} \geq 0, \; z_1, z_2 \in [0,1]\Big\}\cap \cl(\Ri_6),
%\end{align*}
%where
%$$
%\Ri_6:=\Big\{(x,X,z): \; z_1 \leq z_2, \;  X_{12} z_1  > x_1 x_2, \; X_{22} x_1 \geq X_{12} x_2\Big\}.
%$$

% i = 7
\item [IV.] If $i=7$, then
\begin{align*}
\cl(\tilde \S \cap \Ri_i) =\Big\{(x,X,z): \; & X_{11}\geq x^2_1, \; X_{22} \geq \frac{x^2_2}{z_2}, \;(X_{11}-x^2_1) (X_{22}-x^2_2) \geq (X_{12}-x_1 x_2)^2,\; x_1, x_2 \geq 0, \nonumber\\
& z_1, z_2 \leq 1\Big\} \cap \cl(\Ri_7),
\end{align*}
where
\begin{align*}
\Ri_7:=& \Big\{(x,X,z): X_{12} z_1 z_2 < x_1 x_2 (z_1 + z_2 -1), X_{12} >0, z_1, z_2>0, x_1^2 (x_2^2 - X_{22} (1 - z_1)) (X_{22} z_2-x_2^2)\\
 & >  2 x_1 x_2 X_{12} z_1 (X_{22} z_2-x_2^2)- X_{12}^2 (X_{22}(z_1+z_2-1) + x_2^2 (1 - 2 z_1 - z_2(1 -z_1)))\Big\}.
\end{align*}

% i = 8
\item [V.] If $i=8$, then
\begin{align*}
\cl(\tilde \S \cap \Ri_i) =  \Big\{&(x,X,z): \; X_{11} \geq \frac{x^2_1}{z_1}, \; X_{22} \geq \frac{x^2_2}{z_2}, \; x_1, x_2 \geq 0, \; z_1, z_2 \leq 1,\nonumber \\
&z_1 (1-z_2)\Big(X_{11}-\frac{x^2_1}{z_1}\Big) x^2_2 \geq (z_1+z_2-1)\Big(X_{12}\frac{z_1 z_2}{W}-x_1 x_2\Big)^2 \Big\} \cap \cl(\Ri_8),
\end{align*}
where
\begin{align*}
\Ri_8:=&\Big\{(x,X,z): \; X_{12} z_1 z_2 < x_1 x_2 (z_1 + z_2 -1), \; X_{12} >0, \; z_1, z_2 >0\Big\}\setminus (\Ri_6 \cup \Ri_7),
\end{align*}
and
$$
W := (z_1+ z_2-1)- \frac{1}{x_2}\sqrt{(X_{22}z_2-x^2_2)(1-z_1)(z_1+z_2-1)}.
$$
\end{itemize}
%Moreover, the sets $\Ri_i$, $i \in \{0,\ldots, 9\}$ satisfy
%$$\Ri_i \cap \Ri_j = \emptyset, \; \forall i \neq j \in [n], \quad \bigcup_{i=0}^9{\Ri_i} \supseteq \overline{\conv}(\S_2).$$
\end{theorem}

The proof of Theorem~\ref{theTh} is given in Section~\ref{proof1}.

\vspace{0.1in}

Clearly, the convex hull of $\S_2$ has a complicated structure and is not amenable to conic solvers; however, as we detail in the next section, it can be easily incorporated in branch-and-cut based global solvers~\cite{IdaNick18,antigone14,BestuzhevaEtal2021ZR} for generating strong cutting planes that enhance the quality of existing relaxations.

Given a convex set $\C$, throughout this paper we say that the inequality $q(x) \geq 0$ \emph{supports $\C$ at $\hat x$}, or is \emph{a supporting inequality
for $\C$ at $\hat x $}, if $q(x) \geq 0$ for all $x \in \C$
and $\{x: q(x) = 0\}$ is a supporting hyperplane of $\C$ at $\hat x$.

\begin{remark}\label{rem1}
Let us comment on the importance of the set $\tilde \S$ introduced in the statement of Theorem~\ref{theTh}.
While $\overline{\conv}(\S_2)$ could be described without introducing $\tilde \S$, in order to devise our separation algorithm we need the stronger result stated in Theorem~\ref{theTh}. Indeed, to generate supporting inequalities of $\overline{\conv}(\S_2)$ in our separation algorithm, it is important to cover $\overline{\conv}(\S_2)$ with sets of the form $\cl(\tilde \S \cap \Ri_i)$. Since by assumption $\ri(\conv(\S_2)) \subseteq \tilde \S \subseteq \overline{\conv}(\S_2)$, from Lemma~\ref{bps} it follows that:
$$
\partial (\cl(\tilde \S \cap \Ri_i)) \setminus \partial \Ri_i \subset \partial \overline{\conv}(\S_2).
$$
This important property enables us to locate boundary points of $\overline{\conv}(\S_2)$ at which the corresponding supporting hyperplane serves as a separating inequality.
\end{remark}

%While implementing a separation algorithm and testing its impact on global solvers is beyond the scope of this paper, in Section~\ref{separate}, we outline a simple separation algorithm over the closure of the convex hull of $\S_2$.

We conclude this section by presenting a simple yet strong class of valid inequalities for Problem~\eqref{miqp2}, which are a consequence of Theorem~\ref{theTh}.

\begin{proposition}\label{usefull}
Consider the set $\S_n$ defined by~\eqref{TheSet}. Then for any  $i \neq j \in [n]$, any
supporting inequality of
\begin{equation}\label{valid1}
\left\{(x_i, x_j, X_{ii}, X_{ij}, X_{jj}, z_i):
\begin{pmatrix}
z_i   & x_i    & x_j \\
x_i & X_{ii} & X_{ij} \\
x_j & X_{ij} & X_{jj}
\end{pmatrix} \succeq 0 \right\},
\end{equation}
at a boundary point of~\eqref{valid1} satisfying
%$(\tilde x_i, \tilde x_j, \tilde X_{ii}, \tilde X_{ij}, \tilde X_{jj}, \tilde z_i)$
\begin{equation}\label{dom1}
X_{ij} z_i > x_i x_j, \quad X_{jj} x_i > X_{ij} x_j, \quad x_i, x_j > 0, \quad  z_i, z_j \in (0,1),
\end{equation}
is a supporting inequality for $\overline{\conv}(\S_2)$, and hence is valid inequality for $\overline{\conv}(\S_n)$.
\end{proposition}

The proof of Proposition~\ref{usefull} is given at the end of Section~\ref{separate}.

\begin{remark}\label{rem2}
In~\cite{atagom19}, the authors propose the following convex relaxation for $\overline{\conv}(\S_2)$:
\begin{equation}\label{rank1}
\left\{(x_i, x_j, X_{ii}, X_{ij}, X_{jj}, z_i, z_j):
\begin{pmatrix}
z_i+ z_j   & x_i    & x_j \\
x_i & X_{ii} & X_{ij} \\
x_j & X_{ij} & X_{jj}
\end{pmatrix} \succeq 0 \right\},
\end{equation}
and demonstrate that the addition of such constraints to Problem~\eqref{perspR} improves the quality of convex relaxations for Problem~\eqref{miqp}.
By Proposition~\ref{usefull}, any supporting inequality of set~\eqref{valid1} at a point satisfying~\eqref{dom1}
implies a supporting inequality of set~\eqref{rank1}.
\end{remark}

%{\color{red} Fix this remark.}

 The incorporation of the proposed relaxations in a branch-and-cut based global solver and performing an extensive computational study is a subject of future research.

\section{The separation algorithm}
\label{separate}

In this section, by leveraging on our convex hull characterization stated in Theorem~\ref{theTh}, we present a separation algorithm over the closure of the convex hull of $\S_2$.
This algorithm can readily be incorporated in branch-and-cut based global solvers~\cite{IdaNick18,antigone14,BestuzhevaEtal2021ZR} to improve the quality of existing relaxations for nonconvex
problems containing quadratic sets with indicator variables. Define the set
\begin{align}\label{initR}
\tilde\C:=\Big\{(x,X,z):\; & X_{11} \geq \frac{x^2_1}{z_1}, \; X_{22} \geq \frac{x^2_2}{z_2}, \; (X_{11}-x^2_1)(X_{22}-x^2_2) \geq (X_{12}-x_1 x_2)^2, \; X_{12} \geq 0, \nonumber\\
& x_1, x_2 \geq 0, \; z_1, z_2 \in [0,1]\Big\}.
\end{align}
From Proposition~\ref{n=1} and Lemma~\ref{lem1} it follows that $\tilde \C$ is a convex relaxation of $\S_2$.
%Moreover, $\tilde \C \subset \C$, where $\C$ is defined by~\eqref{simple}. Therefore, by Part~(I) of Theorem~\ref{theTh} we deduce that if $(\tilde x, \tilde X, \tilde z) \in \tilde \C$
We start by defining the separation problem:

\paragraph{The separation problem:} Given a point  $(\tilde x, \tilde X, \tilde z) \in \tilde \C$, decide whether $(\tilde x, \tilde X, \tilde z)$ is in $\overline{\conv}(\S_2)$ or not, and in the latter case, find a supporting inequality for $\overline{\conv}(\S_2)$ that is violated by $(\tilde x, \tilde X, \tilde z)$.
\vspace{0.1in}

%In order to construct supporting hyperplanes of $\overline{\conv}(\S_2)$, we make use of the following lemma.

The outline of our separation algorithm is as follows.

\bigskip

\begin{algorithm}[H]
%\SetAlgoLined
%\DontPrintSemicolon
\SetKwFunction{sep}{Separate}
\SetAlgorithmName{\sep}{\sep}{}
\KwIn{A point $(\tilde x, \tilde X, \tilde z) \in \tilde \C$}
\KwOut{A Boolean {\tt Inside}
together with a supporting inequality $h(x,X,z) \geq 0$ for $\overline{\conv}(\S_2)$, violated by $(\tilde x, \tilde X, \tilde z)$ if {\tt Inside = false}.}
Find the unique  $k \in \{1,\ldots,8\}$ such that $(\tilde x, \tilde X, \tilde z) \in \Ri_k$\\
%Check if the point is in or out: \\
\If {all inequalities defining $\cl(\tilde \S \cap \Ri_k)$ are satisfied by $(\tilde x, \tilde X, \tilde z)$,} {
{\tt Inside = true}\\
\Return
}
\Else{
 {\tt Inside = false}\\
% Choose an inequality in $\A_k \setminus \cl(\Ri_k)$ that is violated by $(\tilde x, \tilde X, \tilde z)$\;
% Goto Step 3.

% {\bf Step 3. Generate a separating inequality $h(x, X, z) \geq 0$:}\\
 \If {$k \in \{3,4\}$,}{
     Let $q(x, X, z) = \Big(X_{11} - \frac{x^2_1}{z_2}\Big)\Big(X_{22}- \frac{x^2_2}{z_2}\Big) - \Big(X_{12}- \frac{x_1 x_2}{z_2}\Big)^2$ \\
 }
 \ElseIf {$k = 5$,}{
    Let $q(x, X, z) = \Big(X_{11} - \frac{x^2_1}{z_1}\Big)\Big(X_{22} - \frac{x^2_2}{z_1}\Big) - \Big(X_{12}  -\frac{x_1 x_2}{z_1}\Big)^2$ \\
 }
 \ElseIf{$k = 8$,}{
    Let $q(x, X, z) = z_1 (1-z_2)\Big(X_{11}-\frac{x^2_1}{z_1}\Big) x^2_2 - (z_1+z_2-1)\Big(X_{12}\frac{z_1 z_2}{W}-x_1 x_2\Big)^2$\\
 }
 Define the point $(\hat x, \hat X, \hat z)$ with components equal to $(\tilde x, \tilde X, \tilde z)$ except for $\hat X_{11}$ which is chosen so that
 $q(\hat x, \hat X, \hat z)  =0$. Let $h(x, X, z)$ be the first-order Taylor expansion of $q(x, X, z)$ at $(\hat x, \hat X, \hat z)$.\\
 \Return $h(x, X, z) \geq 0$
 }
\caption{A separation algorithm over $\overline{\conv}(\S_2)$}
\label{alg: sep}
\end{algorithm}

\bigskip

The following proposition establishes the correctness of
our separation algorithm.

\begin{proposition}\label{sepProp}
  Algorithm~\sep solves the separation problem over $\overline{\conv}(\S_2)$.
\end{proposition}

\begin{proof}
From the proof of Theorem~\ref{theTh} it follows that $\cup_{i=1}^8{\Ri_i} \supseteq \tilde \C$. Moreover, since by construction
$\Ri_i \cap \Ri_j = \emptyset$ for all $i \neq j$, we conclude that there exists a unique $k \in \{1,\ldots,8\}$ such that $(\tilde x, \tilde X, \tilde z) \in \Ri_k $.
If all inequalities defining $\cl(\tilde \S \cap \Ri_k)$ are satisfied, then $(\tilde x, \tilde X, \tilde z) \in \overline{\conv}(\S_2)$,
the algorithm sets {\tt Inside = true}, and terminates.
Otherwise, there exists an inequality defining $\cl(\tilde \S \cap \Ri_k)$ that is violated by $(\tilde x, \tilde X, \tilde z)$.
By contradiction, assume that $k \in \{1,2,6,7\}$.
Since by assumption we have $(\tilde x, \tilde X, \tilde z) \in \tilde C$, and from the descriptions of $\cl(\tilde \S \cap \Ri_k)$ and $\Ri_k$ it holds $\tilde C\cap \cl(\Ri_k) \subseteq \cl(\tilde \S \cap \Ri_k)$, it follows the contradiction
$(\tilde x, \tilde X, \tilde z) \in \cl(\tilde \S \cap \Ri_k)$. Therefore,
it suffices to consider $k \in \{3,4,5,8\}$. It can be checked that for any point with $\tilde z_1 = 0$ or $\tilde z_2 = 0$, we have $(\tilde x, \tilde X, \tilde z) \in \Ri_1$.
Henceforth, we can assume that $\tilde z_1, \tilde z_2 > 0$.

First suppose that $k \in \{3,4\}$; since $(\tilde x, \tilde X, \tilde z) \in \tilde \C$, it follows that in this case the inequality
\begin{equation}\label{hope2}
\Big(X_{11} - \frac{x^2_1}{z_2}\Big)\Big(X_{22}- \frac{x^2_2}{z_2}\Big) - \Big(X_{12}- \frac{x_1 x_2}{z_2}\Big)^2 \geq 0
\end{equation}
must be violated by $(\tilde x, \tilde X, \tilde z)$. Let us denote the above inequality by $q(x,X,z) \geq 0$.
Now consider a point $(\hat x, \hat X, \hat z)$ whose components are equal to $(\tilde x, \tilde X, \tilde z)$ except for $\hat X_{11}, \hat X_{22}$, which are chosen as follows.
If  $\tilde X_{22}> \frac{\tilde x^2_2}{\tilde z_2}$, we set $\hat X_{22} = \tilde X_{22}$, while if
$\tilde X_{22}= \frac{\tilde x^2_2}{\tilde z_2}$, we set $\hat X_{22} = \tilde X_{22}+\varepsilon$ for some $\varepsilon > 0$ such that the defining inequalities of $\Ri_k$ involving $X_{22}$ remain satisfied. Notice that such an $\varepsilon$ always exists as all inequalities involving $X_{22}$ are strict.
Next, we set $\hat X_{11}$ so that $q(\hat x, \hat X, \hat z) = 0$. Since by assumption $(\tilde x, \tilde X, \tilde z) \in \Ri_k$ and since the inequalities defining $\Ri_k$ do not contain the variable $X_{11}$, we conclude that $(\hat x, \hat X, \hat z) \in \Ri_k$.

We now use Lemma~\ref{bps} with $\mathcal A=\{(x,X,z): z_1, z_2 > 0\}$, $\Ri= \Ri_k$ and $\C = \tilde \S$, to argue that there exists $r>0$ such that
\begin{equation}\label{supp}
\partial \overline{\conv}(\S_2) \cap \B_r(\hat x, \hat X, \hat z) = \{(x,X,z): q(x,X,z)=0\} \cap \B_r(\hat x, \hat X, \hat z).
\end{equation}
First note at any point with $z_1, z_2 > 0$ all inequalities defining $\tilde \S \cap \Ri_k$ are continuously differentiable. Moreover, after a possibly small perturbation, one can ensure that all inequalities defining $\Ri_k$ are strictly satisfied at $(\hat x, \hat X, \hat z)$. The validity of assumption~\eqref{ass} of Lemma~\ref{bps} follows from the fact that for every $s > 0$ and every point
$(x_1, x_2, X_{11}, X_{12}, X_{22}, z_1, z_2)\in \tilde \S\cap \Ri_k$ such that $q(x,X,z)=0$, we have $(x_1, x_2, X_{11}-\frac{s}{2}, X_{12}, X_{22}, z_1, z_2)\in \B_s(\hat x, \hat X, \hat z) \setminus \tilde \S \cap \Ri_k$. Moreover, $\nabla q(\hat x, \hat X, \hat z) \neq 0$, because
$\frac{\partial q}{\partial X_{11}} (\hat x, \hat X, \hat z)= \hat X_{22}-\frac{\hat x^2_2}{\hat z_2} > 0$.
Therefore, using the fact that $\partial \overline{\conv}(\S_2) = \partial \tilde \S$, relation~\eqref{supp} follows.
Denoting by $h(x, X, z)$ the first-order Taylor expansion of $q(x, X, z)$ at $(\hat x, \hat X, \hat z)$, we conclude that $h(x, X, z) \geq 0$ is a supporting inequality for
$\overline{\conv}(\S_2)$ that is violated by $(\tilde x, \tilde X, \tilde z)$.
%Note that the first-order Taylor expansion exists since $\hat z_2 >0$.

Similarly, if $k =5$, then the inequality
$$
\Big(X_{11} - \frac{x^2_1}{z_1}\Big)\Big(X_{22} - \frac{x^2_2}{z_1}\Big) - \Big(X_{12}  -\frac{x_1 x_2}{z_1}\Big)^2 \geq 0,
$$
must be violated and if $k = 8$, then the inequality
$$
z_1 (1-z_2)\Big(X_{11}-\frac{x^2_1}{z_1}\Big) x^2_2 - (z_1+z_2-1)\Big(X_{12}\frac{z_1 z_2}{W}-x_1 x_2\Big)^2 \geq 0,
$$
must be violated. Since none of inequalities defining $\Ri_5$ or $\Ri_8$ contain
variable $X_{11}$, we can use the method described above to generate supporting inequalities for $\overline{\conv}(\S_2)$ that are violated by $(\tilde x, \tilde X, \tilde z)$.
\end{proof}

% for k = 5, from the two inequalities defining R5, we conclude that X_{22} - \frac{x^2_2}{z_1} is strictly satisfied
%for k = 8, z_2 < 1 (since z2 = 1 is in R6 and R7) and by definition of R8, we have x2 > 0.

We conclude this section by proving Proposition~\ref{usefull}, which follows from Theorem~\ref{theTh} and Proposition~\ref{sepProp}.

\paragraph{Proof of Proposition~\ref{usefull}.}
Denote by $(\tilde x_i, \tilde x_j, \tilde X_{ii}, \tilde X_{ij}, \tilde X_{jj}, \tilde z_i)$ a boundary point of set~\eqref{valid1} satisfying inequalities~\eqref{dom1}.
By letting $i=1$ and $j =2$, inequalities~\eqref{dom1} imply that $(\tilde x_i, \tilde x_j, \tilde X_{ii}, \tilde X_{ij}, \tilde X_{jj}, \tilde z_i)\in \Ri_5$, where  $\Ri_5$ is defined in Part~(III) of Theorem~\ref{theTh}. Hence, by Lemma~\ref{bps}, the remaining inequalities in the description of $\cl(\tilde \S \cap \Ri_5)$
define the boundaries of $\overline{\conv}(\S_2)$:
\begin{equation}\label{bpoints}
X_{ii} \geq \frac{x^2_i}{z_i}, \quad \Big(X_{ii} - \frac{x^2_i}{z_i}\Big)\Big(X_{jj} - \frac{x^2_j}{z_i}\Big) \geq \Big(X_{ij}  -\frac{x_i x_j}{z_i}\Big)^2.
\end{equation}
By Lemma~\ref{lem1}, the convex set~\eqref{valid1} can be equivalently written as inequalities~\eqref{bpoints} together with the inequality
\begin{equation}\label{rbnd}
X_{jj} \geq \frac{x^2_j} {z_i}.
\end{equation}
From $\tilde X_{ij} \tilde z_i > \tilde x_i \tilde x_j$ and $\tilde X_{jj} \tilde x_i > \tilde X_{ij} \tilde x_j$, it follows that
$\tilde X_{jj} > \frac{\tilde x^2_j}{\tilde z_i}$, implying that inequality~\eqref{rbnd} does not define the boundary of~\eqref{valid1} at
$(\tilde x_i, \tilde x_j, \tilde X_{ii}, \tilde X_{ij}, \tilde X_{jj}, \tilde z_i)$. Hence, any supporting inequality of set~\eqref{valid1} at a boundary point satisfying~\eqref{dom1} is a supporting inequality for $\overline{\conv}(\S_2)$, and is a valid inequality for $\overline{\conv}(\S_n)$,
as $\overline{\conv}(\S_n) \subset \{(x,X,z): (x_i, x_j, X_{ii}, X_{ij}, X_{jj}, z_i, z_j) \in \overline{\conv}(\S_2), \; \forall i \neq j \in [n]\}$.

%\section{Numerical Experiments}

\section{Proof of Theorem~1}
\label{proof1}

The purpose of this section is to prove Theorem~\ref{theTh}. Due to the length of the proof, we will redefine some notations used in the statement
of Theorem~\ref{theTh} along the way.
We start by obtaining an extended formulation for the convex hull of $\S_2$ using a standard disjunctive programming technique~\cite{rf70,b85,cs99}. Subsequently, by projecting out the auxiliary variables, we obtain
an explicit description of the convex hull in the space of original variables. This general framework has been utilized to characterize the convex envelope of nonconvex
functions~\cite{ks:11mp,ks:11jogo}, as well as the convex hull of disjunctive sets~\cite{bona15,nguyen18}.

\subsection{Disjunctive Formulation}
Let us partition the nonconvex set $\S_2$ defined by~\eqref{DudeSet} as
$$\S_2 = \P_1 \cup \P_2 \cup \P_3 \cup \P_4,$$
where each $\P_i$, $i \in \{1,\ldots,4\}$, is obtained from $\S_2$ by assigning different binary values to $(z_1, z_2)$:
\begin{align*}
& \P_1 := \{(x,X,z):  z_1 = z_2 = 0, \; x_1=x_2 = X_{11}=X_{12}=X_{22} = 0\},  \\
& \P_2 := \{(x,X,z): z_1 = 1, z_2 = 0, \; X_{11}= x^2_1, \; x_2=X_{12}=X_{22}=0,\; x_1 \geq 0\}, \\
& \P_3 := \{(x,X,z): z_1 = 0, z_2 = 1, \; x_1=X_{11}= X_{12} =0, \; X_{22}=x^2_2, \; x_2 \geq 0\},  \\
& \P_4 := \{(x,X,z): z_1 = z_2 = 1, \; X_{11}= x^2_1, \; X_{12} =x_1 x_2, \; X_{22}=x^2_2, \; x_1, x_2 \geq 0\}.
\end{align*}
To construct the closure of the convex hull of $\S_2$, we employ a sequential approach, where we first convexify each $\P_i$, $i \in \{1,\ldots,4\}$ and then using Lemma~\ref{lemClose}, we obtain
\begin{equation}\label{seq1}
\overline{\conv}(\S_2) = \overline{\conv} \Big(\overline{\conv}(\P_1) \cup \overline{\conv}(\P_2) \cup \overline{\conv}(\P_3)\cup \overline{\conv}(\P_4)\Big).
\end{equation}
Moreover, it can be checked that:
\begin{align}\label{seq2}
& \overline{\conv}(\P_1) = \{(x,X,z):  z_1 = z_2 = 0, \; x_1=x_2 = X_{11}=X_{12}=X_{22} = 0\}, \nonumber\\
& \overline{\conv}(\P_2) = \{(x,X,z): z_1 = 1, z_2 = 0, \; X_{11} \geq x^2_1, \; x_2=X_{12}=X_{22}=0,\; x_1 \geq 0\}, \nonumber\\
& \overline{\conv}(\P_3) = \{(x,X,z): z_1 = 0, z_2 = 1, \; x_1=X_{11}= X_{12} =0, \; X_{22} \geq x^2_2, \; x_2 \geq 0\},  \\
& \overline{\conv}(\P_4) = \Bigg\{(x, X, z): z_1 = z_2 = 1, \; \begin{pmatrix}
1   & x_1    & x_2 \\
x_1 & X_{11} & X_{12} \\
x_2 & X_{12} & X_{22}
\end{pmatrix} \succeq 0, x_1 \geq 0, x_2 \geq 0, X_{12} \geq 0\Bigg\},\nonumber
\end{align}
where to construct $\overline{\conv}(\P_4)$, we made use of a well-known result stating that for $n \leq 4$, the cone of doubly nonnegative matrices
coincides with the cone of completely positive matrices (see for example~\cite{burer15}).
By Lemma~\ref{lem1}, the following is an equivalent description of $\overline{\conv}(\P_4)$, which we will use for our subsequent
derivations:
\begin{align}\label{seq3}
\overline{\conv}(\P_4) = \Big\{ & (x,X,z): z_1 = z_2 = 1, \; X_{11} \geq x^2_1, \; X_{22} \geq x^2_2, \; \nonumber\\
& (X_{11}-x^2_1)(X_{22}-x^2_2) \geq (X_{12}-x_1 x_2)^2, x_1 \geq 0, \; x_2 \geq 0, \; X_{12} \geq 0\Big\}.
\end{align}

\subsection{A convex extended formulation for the convex hull}

The convex hull of $\S_2$ can be obtained by taking the convex hull of the union of four convex sets as defined by~\eqref{seq1} and~\eqref{seq2}.
Using the standard disjunctive programming technique~\cite{rf70,b85}, the closure of the convex hull of $\S_2$  is given by:
$$\overline{\conv}(\S_2) = \cl\Big\{(x,X,z): \; \exists (x,X,z,\tilde x, \tilde X,\lambda) \in  \Sigma \Big\},$$
where
$\tilde x = (\tilde x^1,\ldots, \tilde x^4)$, $\tilde X = (\tilde X^1,\ldots, \tilde X^4)$, with $\tilde x^i =(\tilde x_1^i, \tilde x_2^i)\in \R^2$ and
 $\tilde X^i =(\tilde X_{11}^i, \tilde X_{12}^i, \tilde X_{22}^i)\in \R^3$ for all $i \in \{1,\ldots,4\}$,
and
\begin{align}\label{disj}
\Sigma := \Big\{&(x,X,z,\tilde x, \tilde X,\lambda): x = \sum_{i=1}^4{\tilde x^i}, \; X = \sum_{i=1}^4{\tilde X^i}, \; z = \sum_{i=1}^4{\lambda_i z^i}, \; \sum_{i=1}^4{\lambda_i} = 1, \; \lambda_i \geq 0 \; \forall i \in \{1,\ldots,4\},  \nonumber\\
&\Big(\frac{\tilde x^i}{\lambda_i}, \frac{\tilde X^i}{\lambda_i}, z^i \Big) \in \overline{\conv}(\P_i)\; {\rm if} \;
\lambda_i >0, \; \tilde x^i = \tilde X^i = 0 \;  {\rm if} \;
\lambda_i =0\Big\}.
\end{align}
It is important to note that $\Sigma$ is a convex set;  indeed, the convexity of $\Sigma$ follows from the fact that the inverse image of a convex set under the perspective function is convex (see for example Section~2.3.3 of~\cite{bv:04}).
Substituting~\eqref{seq2} and~\eqref{seq3} in~\eqref{disj}, we get:
\begin{align}
\Sigma = \Big\{&(x,X,z,\tilde x, \tilde X,\lambda):\nonumber\\
&\lambda_1+  \lambda_2 +  \lambda_3 +  \lambda_4 =1, \; \lambda_i \geq 0 \; \forall i \in \{1,\ldots,4\}, \; z_1 = \lambda_2 + \lambda_4, \; z_2 = \lambda_3 + \lambda_4, \nonumber\\
& x_1 = \tilde x^2_1 + \tilde x^4_1, \; X_{11} = \tilde X^2_{11} + \tilde X^4_{11}, \;
x_2 = \tilde x^3_2 + \tilde x^4_2, \; X_{22} = \tilde X^3_{22} + \tilde X^4_{22}, \; X_{12} = \tilde X^4_{12}\nonumber\\
&\begin{cases}
\lambda_2 \tilde X^2_{11} \geq (\tilde x^2_1)^2, \; \tilde x^2_1 \geq 0, \qquad & {\rm if}\; \lambda_2 \neq 0\nonumber\\
\tilde x^2_1 = \tilde X^2_{11} = 0,  \qquad \qquad \qquad &{\rm otherwise}
\end{cases}\\
&\begin{cases}
\lambda_3 \tilde X^3_{22} \geq (\tilde x^3_2)^2, \; \tilde x^3_2 \geq 0, \qquad & {\rm if}\; \lambda_3 \neq 0\nonumber\\
\tilde x^3_2 = \tilde X^3_{22} = 0,  \qquad \qquad \qquad & {\rm otherwise}
\end{cases}\\
&\begin{cases}
& \lambda_4 \tilde X^4_{11} \geq (\tilde x^4_1)^2, \; \lambda_4 \tilde X^4_{22} \geq (\tilde x^4_2)^2, \; (\lambda_4\tilde X^4_{11}-(\tilde x^4_1)^2)(\lambda_4 \tilde X^4_{22}-(\tilde x^4_2)^2) \geq (\lambda_4 \tilde X^4_{12}-\tilde x^4_1 \tilde x^4_2)^2,\nonumber\\
&\tilde x^4_1 \geq 0, \; \tilde x^4_2 \geq 0, \; \tilde X^4_{12} \geq 0, \qquad \qquad \qquad {\rm if}\; \lambda_4 \neq 0\nonumber\\
& \tilde x^4_1 =\tilde x^4_2 = \tilde X^4_{11} = \tilde X^4_{12}=\tilde X^4_{22} = 0,  \; \qquad {\rm otherwise}
\end{cases}\nonumber\\
\Big\}.\label{sbar}
\end{align}
%
%\paragraph{A simple face of $\overline{\conv}(\S_2)$.}
Before proceeding to the projection of $\Sigma$, we characterize a face of $\overline{\conv}(\S_2)$ defined as
\begin{equation}\label{fface}
\F:=\{(x,X,z) \in \overline{\conv}{(\S_2)}:\;  X_{12} = 0\}.
\end{equation}
This will in turn simplify our subsequent derivations.
%First, consider the face of $\overline{\conv}(\S_2)$ defined by $z_1 = z_2= 0$;
%it is simple to check that this face is defined by $\{(x,X,z): x = X= z = 0\}$.
%Now, consider the face $\F_1$ of $\overline{\conv}(\S_2)$ defined by $z_1 = 0$.
%For any point in $\Sigma$ with $z_1 = 0$ and $z_2 \neq 0$, we have $\lambda_2 = \lambda_4 =0$; this in turn implies that
%$\tilde x^2_1 = \tilde X^2_{11} = \tilde x^4_1 =\tilde x^4_2 = \tilde X^4_{11} = \tilde X^4_{12}=\tilde X^4_{22} = 0$,
%and as a result $x_1 = X_{11} = X_{12} = 0$. Hence, $\lambda_3 = z_2$, $\tilde x^3_2 = x_2$, and $\tilde X^3_{22}=X_{22}$, implying
%\begin{align}\label{face1}
%\cl(\conv(\S_2) \cap \{(x,X,z): z_{1} = 0\})= \cl\Big\{(x,X,z): & \; x_1 = X_{11} = X_{12} = z_1 = 0, \; X_{22} z_2 \geq x^2_2,\; x_2 \geq 0, \nonumber\\
%& z_2 \in (0,1]\Big\}.
%\end{align}
%Using a similar line of arguments we obtain
%\begin{align}\label{face2}
%\cl(\conv(\S_2) \cap \{(x,X,z): z_{2} = 0\})= \cl\Big\{(x,X,z): & \; x_2 = X_{22} = X_{12} = z_2 = 0, \; X_{11} z_1 \geq x^2_1,\; x_1 \geq 0, \nonumber\\
%& z_1 \in (0,1]\Big\}.
%\end{align}
%Consider the face $\F_2$ of $\overline{\conv}(\S_2)$ defined by $X_{12}= 0$.
First, from Proposition~\ref{n=1} it follows that
$\F \subseteq \C$, where
$$\C := \Big\{(x,X,z): X_{11} z_1 \geq x^2_1, \; X_{22} z_2 \geq x^2_2,\; X_{12} =0, \; x_1,x_2 \geq 0, \; X_{11}, X_{22} \geq 0, \; z_1,z_2 \in [0,1]\Big\}.$$
Second, letting $X_{12} =0$ and $\lambda_4 = 0$ in~\eqref{sbar} gives $\C \subseteq \F$, implying $\F = \C$. We have thus shown
\begin{equation}\label{face3}
\overline{\conv}(\S_2) \cap \{(x,X,z): X_{12} = 0\} = \C.
\end{equation}
%
%\vspace{1em}
%\emph{Therefore, by~\eqref{face1},~\eqref{face2}, and~\eqref{face3}, we conclude the proof of Part~(I) of Theorem~\ref{theTh} for $i = 0$.}
%\vspace{1em}
%
Henceforth, by Lemma~\ref{drs}, we can assume that $X_{12} > 0$; this in turn implies $\lambda_4 > 0$.
%In fact, we make use of Lemma~\ref{drs} to further
%simplify the description of $\Sigma$ defined by~\eqref{sbar} by discarding some parts of its boundary.
%
%For any convex set $\D$ such that ${\rm ri}(\C) \subseteq \D \subseteq \cl(\C)$, we have $\cl(\D) = \cl(\C)$ (see for example, Theorem~6.3 in~\cite{rf70}). Together with Lemma~\ref{drs}, this implies that to construct $\overline{\conv}(\S_2)$,
%\begin{itemize}[leftmargin=*]
%\item [(i)] in~\eqref{seq2}, we can replace all inequalities $x_1 \geq 0$, $x_2 \geq 0$, by
%$x_1 > 0$, $x_2 > 0$,
%\item [(ii)] in~\eqref{sbar}, we can replace $\lambda_4 \geq 0$ by $\lambda_4 > 0$.
%\end{itemize}
We then obtain the following description of the closure of the convex hull of $\S_2$:
$$\overline{\conv}(\S_2) = \cl\Big\{(x,X,z): \; \exists (x,X,z,\tilde x, \tilde X,\lambda) \in  \Sigma' \Big\},$$
where
\begin{align}
\Sigma' := \Big\{&(x,X,z,\tilde x, \tilde X,\lambda): X_{12} > 0,\nonumber\\
&\lambda_1+  \lambda_2 +  \lambda_3 +  \lambda_4 =1, \; \lambda_i \geq 0 \; \forall i \in \{1,2,3\}, \; \lambda_4 >0, \; z_1 = \lambda_2 + \lambda_4, \; z_2 = \lambda_3 + \lambda_4, \nonumber\\
& x_1 = \tilde x^2_1 + \tilde x^4_1, \; X_{11} = \tilde X^2_{11} + \tilde X^4_{11}, \;
x_2 = \tilde x^3_2 + \tilde x^4_2, \; X_{22} = \tilde X^3_{22} + \tilde X^4_{22}, \; X_{12} = \tilde X^4_{12}\nonumber\\
&\begin{cases}
\lambda_2 \tilde X^2_{11} \geq (\tilde x^2_1)^2, \; \tilde x^2_1 \geq 0, \qquad & {\rm if}\; \lambda_2 \neq 0\nonumber\\
\tilde x^2_1 = \tilde X^2_{11} = 0,  \qquad \qquad \qquad & {\rm otherwise}
\end{cases}\\
&\begin{cases}
\lambda_3 \tilde X^3_{22} \geq (\tilde x^3_2)^2, \; \tilde x^3_2 \geq 0, \qquad & {\rm if}\; \lambda_3 \neq 0\nonumber\\
\tilde x^3_2 = \tilde X^3_{22} = 0,  \qquad \qquad \qquad & {\rm otherwise}
\end{cases}\\
& \lambda_4 \tilde X^4_{11} \geq (\tilde x^4_1)^2, \; \lambda_4 \tilde X^4_{22} \geq (\tilde x^4_2)^2, \; (\lambda_4\tilde X^4_{11}-(\tilde x^4_1)^2)(\lambda_4 \tilde X^4_{22}-(\tilde x^4_2)^2) \geq (\lambda_4 \tilde X^4_{12}-\tilde x^4_1 \tilde x^4_2)^2,\nonumber\\
&\tilde x^4_1 \geq 0, \; \tilde x^4_2 \geq 0, \; \tilde X^4_{12} \geq 0.\nonumber\\
\Big\}.\label{sbarp}
\end{align}

\paragraph{Projecting out some of the auxiliary variables.}
In the remainder of the proof, our objective is to project out variables $(\tilde x, \tilde X, \lambda)$ from the description of $\Sigma'$, hence obtaining
an explicit description of
$\overline{\conv}(\S_2)$ in the original space. We perform the projection in a number of steps:

\begin{itemize}[leftmargin=*]
\item [(i)] Project out $\lambda_1, \lambda_2, \lambda_3, \tilde X^4_{12}$:
\begin{align*}
\overline{\conv}(\S_2) = &\cl\Big\{(x,X,z): \; \exists (x,X,z,\tilde x, \tilde X,\lambda_4): \; X_{12} > 0,\\
& z_1 + z_2-1 \leq \lambda_4 \leq \min\{z_1, z_2\}, \; \lambda_4 > 0,\\
&x_1 = \tilde x^2_1 + \tilde x^4_1, \; X_{11} = \tilde X^2_{11} + \tilde X^4_{11}, \;
x_2 = \tilde x^3_2 + \tilde x^4_2, \; X_{22} = \tilde X^3_{22} + \tilde X^4_{22}, \\
&\begin{cases}
(z_1-\lambda_4) \tilde X^2_{11} \geq (\tilde x^2_1)^2, \; \tilde x^2_1 \geq 0, \qquad & {\rm if}\; \lambda_4 \neq z_1\\
\tilde x^2_1 = \tilde X^2_{11} = 0,  \quad \qquad\qquad \qquad \qquad & {\rm otherwise}
\end{cases}\\
&\begin{cases}
(z_2-\lambda_4) \tilde X^3_{22} \geq (\tilde x^3_2)^2, \; \tilde x^3_2 \geq 0, \qquad &{\rm if}\; \lambda_4 \neq z_2\\
\tilde x^3_2 = \tilde X^3_{22} = 0,  \quad \qquad\qquad \qquad \qquad &{\rm otherwise}
\end{cases}\\
& \lambda_4 \tilde X^4_{11} \geq (\tilde x^4_1)^2, \; \lambda_4 \tilde X^4_{22} \geq (\tilde x^4_2)^2, \; (\lambda_4\tilde X^4_{11}-(\tilde x^4_1)^2)(\lambda_4 \tilde X^4_{22}-(\tilde x^4_2)^2) \geq (\lambda_4 X_{12}-\tilde x^4_1 \tilde x^4_2)^2,\\
&\tilde x^4_1 \geq 0, \; \tilde x^4_2 \geq 0\\
\Big\}.
\end{align*}

\item [(ii)] Project out $\tilde x^2_1, \tilde X^2_{11}, \tilde x^3_2, \tilde X^3_{22}$:
\begin{align*}
\overline{\conv}(\S_2) = & \cl\Big\{(x,X,z): \; \exists (x,X,z,\tilde x, \tilde X,\lambda_4): \; X_{12} > 0,\\
& z_1 + z_2-1 \leq \lambda_4 \leq \min\{z_1, z_2\}, \; \lambda_4 > 0,\\
%&X_{11} = \tilde X^2_{11} + \tilde X^4_{11}, \; X_{22} = \tilde X^3_{22} + \tilde X^4_{22}, \\
&\begin{cases}
(z_1-\lambda_4) (X_{11}-\tilde X^4_{11}) \geq (x_1-\tilde x^4_1)^2, \; \tilde x^4_1 \leq x_1, \qquad &{\rm if}\; \lambda_4 \neq z_1\\
\tilde x^4_1 = x_1, \; \tilde X^4_{11} = X_{11},  \; \qquad \qquad \qquad \qquad \qquad \qquad &{\rm otherwise}
\end{cases}\\
&\begin{cases}
(z_2-\lambda_4) (X_{22}-\tilde X^4_{22}) \geq (x_2 - \tilde x^4_2)^2, \; \tilde x^4_2 \leq x_2, \qquad &{\rm if}\; \lambda_4 \neq z_2\\
\tilde x^4_2 =x_2, \; \tilde X^4_{22} = X_{22},  \; \qquad \qquad \qquad \qquad \qquad \qquad &{\rm otherwise}
\end{cases}\\
& \lambda_4 \tilde X^4_{11} \geq (\tilde x^4_1)^2, \; \lambda_4 \tilde X^4_{22} \geq (\tilde x^4_2)^2, \; (\lambda_4\tilde X^4_{11}-(\tilde x^4_1)^2)(\lambda_4 \tilde X^4_{22}-(\tilde x^4_2)^2) \geq (\lambda_4 X_{12}-\tilde x^4_1 \tilde x^4_2)^2,\\
&\tilde x^4_1 \geq 0, \; \tilde x^4_2 \geq 0\\
\Big\}.
\end{align*}

\item [(iii)] Project out $\tilde X^4_{11}, \tilde X^4_{22}$:
note that by $\lambda_4 > 0$ and $\lambda_4 \leq \min\{z_1, z_2\}$, we have $z_1, z_2 > 0$.
\begin{align*}
\overline{\conv}(\S_2) = & \cl\Big\{(x,X,z): \; \exists (x,X,z,\tilde x^4_1, \tilde x^4_2,\lambda_4): \; X_{12} > 0\\
& z_1 + z_2-1 \leq \lambda_4 \leq \min\{z_1, z_2\}, \; \lambda_4 > 0,\\
&\begin{cases}
& X_{11}- \frac{(\tilde x^4_1)^2}{\lambda_4}-\frac{(x_1-\tilde x^4_1)^2}{z_1-\lambda_4} \geq 0, \;
X_{22}- \frac{(\tilde x^4_2)^2}{\lambda_4}-\frac{(x_2 - \tilde x^4_2)^2}{z_2-\lambda_4} \geq 0, \\
&\Big(X_{11}-\frac{(\tilde x^4_1)^2}{\lambda_4}-\frac{(x_1-\tilde x^4_1)^2}{z_1-\lambda_4}\Big)
\Big(X_{22}-\frac{(\tilde x^4_2)^2}{\lambda_4 }-\frac{(x_2 - \tilde x^4_2)^2}{(z_2-\lambda_4)}\Big) \geq \Big(X_{12}-\frac{\tilde x^4_1 \tilde x^4_2}{\lambda_4 }\Big)^2,\\
&0 \leq \tilde x^4_1 \leq x_1, \; 0 \leq \tilde x^4_2 \leq x_2,  \qquad\qquad \qquad \qquad \qquad {\rm if}\; \lambda_4 \neq z_1, z_2\\
& X_{11} - \frac{x_1^2}{z_1} \geq 0, \; X_{22}-\frac{(\tilde x^4_2)^2}{z_1}-\frac{(x_2 - \tilde x^4_2)^2}{z_2-z_1} \geq 0, \\
& \Big(X_{11}-\frac{x_1^2}{z_1}\Big)\Big(X_{22}-\frac{(\tilde x^4_2)^2}{z_1}-\frac{(x_2 - \tilde x^4_2)^2}{z_2-z_1}\Big) \geq \Big(X_{12}-\frac{x_1 \tilde x^4_2}{z_1}\Big)^2,\\
&0 \leq \tilde x^4_2 \leq x_2,  \qquad\qquad \qquad \qquad \qquad \qquad \qquad \qquad {\rm if}\; \lambda_4 = z_1, \; z_1 < z_2\\
& X_{11}-\frac{(\tilde x^4_1)^2}{z_2}-\frac{(x_1-\tilde x^4_1)^2}{z_1-\lambda_4} \geq 0, \; X_{22} -\frac{x_2^2}{z_2} \geq 0, \\
&\Big(X_{11}-\frac{(\tilde x^4_1)^2}{z_2}-\frac{(x_1-\tilde x^4_1)^2}{z_1-z_2}\Big)\Big(X_{22}-\frac{x_2^2}{z_2}\Big)
\geq \Big(X_{12}-\frac{\tilde x^4_1  x_2}{z_2}\Big)^2,\\
&0 \leq \tilde x^4_1 \leq x_1, \qquad\qquad \qquad \qquad \qquad \qquad \qquad \qquad {\rm if}\; \lambda_4 = z_2, \; z_2 < z_1\\
&  X_{11}- \frac{x_1^2}{\lambda_4} \geq 0, \; X_{22}-\frac{x_2^2}{\lambda_4}\geq 0, \\
& \Big(X_{11}-\frac{x_1^2}{\lambda_4}\Big)\Big(X_{22}-\frac{x_2^2}{\lambda_4}\Big) \geq \Big(X_{12}-\frac{x_1 x_2}{\lambda_4}\Big)^2,
 \qquad \qquad {\rm if}\; \lambda_4 = z_1 = z_2
\end{cases}\\
\Big\}.
\end{align*}
\end{itemize}
Henceforth, for notational simplicity, whenever we write a function of the form $f(u,v)=\frac{u^2}{v}$, we refer to its closure $\hat f$ as defined by~\eqref{closure}. We use a similar convention when composing $f(u,v)$ by an affine mapping. We then have:
\begin{equation}
\overline{\conv}(\S_2) =  \cl\Big\{(x,X,z): \; \exists (x,X,z,\tilde x^4_1, \tilde x^4_2,\lambda_4) \in \tilde \Sigma\Big\},\label{workSet0}
\end{equation}
where
\begin{align}\label{workSet}
\tilde \Sigma := & \Big\{(x,X,z,\tilde x^4_1, \tilde x^4_2,\lambda_4): X_{12} > 0, \; z_1 + z_2-1 \leq \lambda_4 \leq \min\{z_1, z_2\}, \; \lambda_4 > 0,\nonumber\\
& X_{11}- \frac{(\tilde x^4_1)^2}{\lambda_4}-\frac{(x_1-\tilde x^4_1)^2}{z_1-\lambda_4} \geq 0, \;
X_{22}- \frac{(\tilde x^4_2)^2}{\lambda_4}-\frac{(x_2 - \tilde x^4_2)^2}{z_2-\lambda_4} \geq 0, \nonumber\\
&\Big(X_{11}-\frac{(\tilde x^4_1)^2}{\lambda_4}-\frac{(x_1-\tilde x^4_1)^2}{z_1-\lambda_4}\Big)
\Big(X_{22}-\frac{(\tilde x^4_2)^2}{\lambda_4 }-\frac{(x_2 - \tilde x^4_2)^2}{(z_2-\lambda_4)}\Big) \geq \Big(X_{12}-\frac{\tilde x^4_1 \tilde x^4_2}{\lambda_4 }\Big)^2,\nonumber\\
&0 \leq \tilde x^4_1 \leq x_1, \; 0 \leq \tilde x^4_2 \leq x_2\Big\}.
\end{align}
The set $\tilde \Sigma$ is convex; to see this, note that $\tilde \Sigma$ is the projection of $\Sigma'$ onto the space $(x,X,z,\tilde x^4_1, \tilde x^4_2,\lambda_4)$.
The set $\Sigma'$ is convex as it is obtained by removing some of the faces of the convex set $\Sigma$. Since convexity is preserved under projection, it follows that
$\tilde \Sigma$ is a convex set.

\subsection{A piece-by-piece characterization of the convex hull: the simple piece}

Henceforth, define
$$\tilde \S := \Big\{(x,X,z): (x,X,z,\tilde x^4_1, \tilde x^4_2,\lambda_4) \in \tilde \Sigma\Big\} \cup \F,$$
where $\tilde \Sigma$ and $\F$ are defined by~\eqref{workSet} and~\eqref{fface}, respectively.
The convexity of $\tilde \S$ follows from the convexity of
$\tilde \Sigma$ and the fact that $\F$ defines a face of the projection of $\tilde \Sigma$ onto $(x,X,z)$.
By Lemma~\ref{drs}, we have $\rm ri(\conv(\S_2)) \subseteq \tilde \S \subseteq \overline{\conv}(\S_2)$, implying
$$\overline{\conv}(\S_2) = \cl(\tilde \S).$$
In the remainder of the proof, we define the sets $\Ri_i$, $i \in \{1,\ldots,8\}$, as in the statement of Theorem~\ref{theTh}, satisfying $\Ri_i \cap \Ri_j = \emptyset$ for all $i \neq j$,  and
$\bigcup_{i=1}^8 \Ri_i \supset \overline{\conv}(\S_2)$.
Subsequently, for each $i \in \{1,\ldots,8\}$, we characterize the set $\cl(\tilde \S \cap \Ri_i)$ in the original space;
finally, using Lemma~\ref{intcl}, we obtain a characterization of $\overline{\conv}(\S_2)$ given by:
$$
\overline{\conv}(\S_2) = \bigcup_{i=1}^8{\cl(\tilde \S \cap \Ri_i)}.
$$

We start by characterizing a part of the convex hull with a simple algebraic description.
Consider the set
\begin{equation}\label{simple}
\bar\C := \Big\{(x, X, z): X_{11} \geq \frac{x_1^2}{z_1}, \; X_{22} \geq \frac{x_2^2}{z_2},\; X_{12} \geq 0, \; x_1, x_2 \geq 0, \; z_1, z_2 \in [0,1]\Big\}.
\end{equation}
The set
$\bar\C$ is closed and convex; from Proposition~\ref{n=1} it follows that $\bar\C \supseteq \overline{\conv}(\S_2)$. Now, define:
%$\Ri_1$ as in the statement of Theorem~\ref{theTh}:
%
%\begin{align}\label{rg1}
%\Ri_1 := \Big\{(x, X, z): \; & x_1 x_2 (z_1 + z_2-1) \leq X_{12} z_1 z_2 \leq x_1 x_2 \min\{z_1, z_2\},\nonumber\\
%& \; X_{12} >0, \; x_1 , x_2 >0,\; z_1, z_2 >0  \Big\}.
%\end{align}
\begin{align}\label{rg1}
\Ri'_1 := \Big\{(x, X, z): \; & z_1 \leq z_2, \; x_1 x_2 (z_1 + z_2-1) \leq X_{12} z_1 z_2, \;  X_{12} z_2 \leq x_1 x_2,  \; X_{12} >0, \; x_1, x_2 >0,\nonumber\\
  z_1>0  \Big\}.
\end{align}
%, since the inequality $z_1 + z_2-1 \leq \lambda_4 \leq \min\{z_1, z_2\}$ implies $z_1, z_2 \leq 1$, then
%We are particularly interested in characterizing regions over which $\conv(\S)$ is strictly contained in $\C$.
We claim that
\begin{equation}\label{cle}
\cl(\tilde \S \cap \Ri'_1)=\cl(\bar\C \cap \Ri'_1).
\end{equation}
Indeed, since $\bar\C \cap \Ri'_1 \supseteq \overline{\conv} (\S_2)\cap \Ri'_1=\cl(\tilde \S) \cap \Ri'_1$, it follows that
$\cl(\bar\C \cap \Ri'_1) \supseteq \cl(\cl(\tilde \S) \cap \Ri'_1) \supseteq \cl(\tilde \S \cap \Ri'_1)$.
Hence, to prove~\eqref{cle}, we need to show that $\cl(\bar\C \cap \Ri'_1) \subseteq \cl(\tilde \S \cap \Ri'_1)$. That is, it is suffices to show that
$\bar\C \cap \Ri'_1 \subseteq \tilde \S \cap \Ri'_1$.
%We are particularly interested in characterizing regions over which $\conv(\S)$ is strictly contained in $\C$.
To this aim, observe that every point $(x,X,z,\tilde x^4_1, \tilde x^4_2,\lambda_4)$ satisfying
\begin{equation}\label{thp}
(x,X,z) \in \Ri'_1, \quad  \tilde x^4_1 = \lambda_4 \frac{x_1}{z_1}, \quad \tilde x^4_2 = \lambda_4 \frac{x_2}{z_2}, \quad \lambda_4 = X_{12} \frac{z_1 z_2}{x_1 x_2},
\end{equation}
belongs to $\tilde \Sigma$. Using~\eqref{thp} to substitute $(\tilde x^4_1, \tilde x^4_2, \lambda_4)$ in inequalities~\eqref{workSet} yields:
\begin{align*}
&\Big\{(x,X,z): \; X_{11}-\frac{x_1^2}{z_1} \geq 0, \;
X_{22}- \frac{x_2^2}{z_2}\geq 0, \;X_{12} > 0, \; x_1, x_2 > 0, \; z_1, z_2 \in (0,1] \Big\}
\cap\Ri'_1 \nonumber\\
& = \bar\C \cap \Ri'_1 \subseteq \tilde \S \cap \Ri'_1.
\end{align*}
Hence~\eqref{cle} holds. Finally, we characterize the closure of $\bar\C \cap \Ri'_1$. We claim that
\begin{equation*}
\cl(\bar\C \cap \Ri'_1)= \chi:=\bar\C \cap \Big\{(x, X, z): \;  z_1\leq z_2, \; x_1 x_2 (z_1 + z_2-1) \leq X_{12} z_1 z_2, \; X_{12} z_2 \leq x_1 x_2  \Big\}.
%\\
%&\setminus \Big(\Big\{(x, X, z): \;  z_1=0,z_2>0,X_{12}>0\Big\}\cup \Big\{(x, X, z): \;  z_1>0,z_2=0,X_{12}>0\Big\}\Big).
%\end{split}
\end{equation*}
%We deduce that
%\begin{align*}
%\C &\cap \{x_1 x_2 (z_1 + z_2-1) \leq X_{12} z_1 z_2 \leq x_1 x_2 \min\{z_1, z_2\} \}\\
%&\subset\cl\{X_{11}-\frac{x_1^2}{z_1} \geq 0, \;
%X_{22}- \frac{x_2^2}{z_2}\geq 0, \;X_{12} > 0, \; x_1 > 0, \; x_2 > 0, \; z_1 > 0,\; z_2 > 0\}\\
%&\qquad  \cap \{ x_1 x_2 (z_1 + z_2-1) \leq X_{12} z_1 z_2 \leq x_1 x_2 \min\{z_1, z_2\} \}\\
%&\subset \overline{\conv}(\S)\cap \{x_1 x_2 (z_1 + z_2-1) \leq X_{12} z_1 z_2 \leq x_1 x_2 \min\{z_1, z_2\} \}.
%\end{align*}
%
%We observe that
%\begin{align*}
%\cl &\Big\{(x,X,z): X_{11}-\frac{x_1^2}{z_1} \geq 0, \;
%X_{22}- \frac{x_2^2}{z_2}\geq 0, \; X_{12} > 0, \; x_1 > 0, \; x_2 > 0, \; z_1 > 0, \; z_2 > 0\}\\
%=A:=& \Big\{(x,X,z): X_{11}-\frac{x_1^2}{z_1} \geq 0, \;
%X_{22}- \frac{x_2^2}{z_2}\geq 0, \; X_{12} \geq 0, \; x_1 \geq 0, \; x_2 \geq 0, \; z_1 \geq 0, \; z_2 \geq 0\}.
%\end{align*}
%where the last inclusion follows from \eqref{utile}, while the first inclusion is proved as follows:
%
For every $(x,X,z)\in \chi$, the sequence $(x^n,X^n,z^n)$ defined as
\begin{align*}
&x^n_1=\max\Big\{x_1,\frac{1}{n^2}\Big\}, \; x^n_2=\max\Big\{x_2,\frac{1}{n^2}\Big\},\; X^n_{12}=\max\Big\{X_{12},\frac{1}{n^2}\Big\},\; X^n_{11}=\max\Big\{X_{11}, \frac{1}{n}\Big\}, \\
&X^n_{22}=\max\Big\{X_{22}, \frac{1}{n}\Big\}, \; z^n_1=\max\Big\{z_1,\frac{1}{n^2}\Big\}, \; z^n_2=\max\Big\{z_2,\frac{1}{n^2}\Big\},
\end{align*}
%{\color{red} Letting $x_1 = z_1 = 0$, this sequence does not work because it gives $X^n_{11} \geq 1$.}
%$x^n_1=x_1+1/n$, $x^n_2=x_2+1/n$, $X^n_{12}=X^n_{12}+1/n$, $X^n_{11}=\max\{X_{11}, \frac{(x_1+1/n)^2}{z_1+1/n}\}$, $X^n_{22}=\max\{X_{22}, \frac{(x_2+1/n)^2}{z_2+1/n}\}$, $z^n_1=z_1+1/n$, $z^n_2=z_2+1/n$,
converges to $(x,X,z)$ as $n \to \infty$ and clearly $\{(x^n,X^n,z^n)\}\subset \bar\C \cap \Ri'_1$. Hence $\chi\subset \cl(\bar\C \cap \Ri'_1)$.
The reverse inclusion follows from the fact that $\chi$ is a closed set and $\chi\supset \bar\C \cap \Ri'_1$.

% the set
%$$\C \cap \Big\{(x, X, z): \;  x_1 x_2 (z_1 + z_2-1) \leq X_{12} z_1 z_2 \leq x_1 x_2 \min\{z_1, z_2\} \Big\}$$
%is closed and contains $\C \cap \Ri_1$; hence it also contains $\cl(\C \cap \Ri_1)$. Hence, to prove $\cl(\C \cap \Ri_1)\subset A$ and consequently prove~\eqref{cle2}, it is enough to show that
%$$\Big(\Big\{(x, X, z): \;  z_1=0,z_2>0,X_{12}>0\Big\}\cup \Big\{(x, X, z): \;  z_1>0,z_2=0,X_{12}>0\Big\}\Big)\cap \cl(\C \cap \Ri_1)=\emptyset.$$
%Assume by contradiction the set above is nonempty. By symmetry, we can assume there exists a point $(x,X,z)\in \cl(\C \cap \Ri_1)$ such that $z_1=0,z_2>0,X_{12}>0$. By definition of $\C$ this implies that $x_1=0$. Moreover, there exists a sequence $(x^n,X^n,z^n)\subset \C \cap \Ri_1$ converging to $(x,X,z)$ as $n \to \infty$. Hence, for $n$ large enough, by the definition of $\Ri_1$, we deduce the following contradiction:
%$$0<X_{12} z_2 \leftarrow X_{12}^n z_2^n \leq x_1^n x_2^n \to 0.$$
%Hence, ~\eqref{cle2} holds.

Using a similar line of arguments, we can show that
\begin{equation}\label{cle2}
\cl(\tilde \S \cap \Ri''_1)=\cl(\bar\C \cap \Ri''_1),
\end{equation}
where
\begin{align*}
\Ri''_1 := \Big\{(x, X, z): \; & z_2 \leq z_1, \; x_1 x_2 (z_1 + z_2-1) \leq X_{12} z_1 z_2, \;  X_{12} z_1 \leq x_1 x_2, \; X_{12} >0, \; x_1,x_2 >0,\nonumber\\
&  z_2>0  \Big\}.
\end{align*}

Using the fact that the union of closure of sets equals the closure of union of those sets, \emph{the proof of Part~(I) of Theorem~\ref{theTh} for $i=1$ follows
from~\eqref{face3},~\eqref{cle}, and~\eqref{cle2}}.

\vspace{0.2in}

%
%Moreover, inequality~\eqref{r0}
%implies $z_1, z_2 \leq 1$. Hence, we conclude that over the region defined by~\eqref{r0},
%the closure of the convex hull of $\S$ is given by
%\begin{align}
%\overline{\conv}(\S)=\Big\{(x,X,z): X_{11} \geq \frac{x_1^2}{z_1}, \;
%X_{22} \geq \frac{x_2^2}{z_2}, \; X_{12} \geq 0, \; x_1 \geq 0, \; x_2 \geq 0, z_1 \geq 0, z_2 \geq 0\Big\}. \label{peice0}
%\end{align}
%
%{\color{red} We need to take the closure of the intersection of
%\begin{align*}
%& X_{11}-\frac{x_1^2}{z_1} \geq 0, \;
%X_{22}- \frac{x_2^2}{z_2}\geq 0, \;X_{12} > 0, \; x_1 > 0, \; x_2 > 0, \; z_1 > 0, \; z_2 > 0,
%\end{align*}
%with inequality~\eqref{r0}.
%}

Therefore, to characterize $\overline{\conv}(\S_2)$, it suffices to characterize $\tilde \S$
over the following regions:
\begin{itemize}
\item Region~$\U_1$:
\begin{equation}\label{r1}
\U_1:=\Big\{(x,X,z): \; X_{12} \max \{z_1, z_2\} > x_1 x_2,\; X_{12} > 0, \; z_1, z_2 >0\Big\}
\end{equation}
\item Region~$\U_2$:
\begin{equation}\label{r2}
\U_2:=\Big\{(x,X,z): \; X_{12} z_1 z_2 < x_1 x_2 (z_1 + z_2 -1), \; X_{12} >0,\; z_1, z_2 >0\Big\}
\end{equation}
\end{itemize}

\subsection{Projection by convex optimization}
\label{proof2}

We next show that the projection of $\tilde \Sigma$ as defined by~\eqref{workSet} onto the space of $(x,X,z)$ can be obtained by solving a \emph{parametric convex optimization problem}.
Define the functions
\begin{equation}\label{yyy}
g_1 := X_{11}- \frac{(\tilde x^4_1)^2}{\lambda_4}-\frac{(x_1-\tilde x^4_1)^2}{z_1-\lambda_4}, \quad g_2 := X_{22}- \frac{(\tilde x^4_2)^2}{\lambda_4}-\frac{(x_2 - \tilde x^4_2)^2}{z_2-\lambda_4} , \quad h   := X_{12}-\frac{\tilde x^4_1 \tilde x^4_2}{\lambda_4}.
\end{equation}
We claim that $\tilde \Sigma$, which by definition is the set of points satisfying:
\begin{align}\label{sys1}
\tag{System I}
& g_1 g_2 \geq h^2, \quad g_1 \geq 0, \quad g_2 \geq 0,\nonumber\\
& z_1 + z_2-1 \leq \lambda_4 \leq \min\{z_1, z_2\}, \quad \lambda_4 > 0,\nonumber\\
&0 \leq \tilde x^4_1 \leq x_1, \quad 0 \leq \tilde x^4_2 \leq x_2, \quad X_{12} > 0, \nonumber
\end{align}
coincides with the set of points satisfying:
\begin{align}
\label{sys2}
\tag{System II}
& g_1 \geq {\rm cl}\Big(\frac{h^2}{g_2}\Big), \quad g_2 \geq 0,\nonumber\\
& z_1 + z_2-1 \leq \lambda_4 \leq \min\{z_1, z_2\}, \quad \lambda_4 > 0,\nonumber\\
&0 \leq \tilde x^4_1 \leq x_1, \quad 0 \leq \tilde x^4_2 \leq x_2, \quad X_{12} > 0, \nonumber
\end{align}
where as in \eqref{closure} we define
\begin{align} \label{defcl}
{\rm cl}\Big(\frac{h^2}{g_2}\Big) :=
\begin{cases}
\frac{h^2}{g_2}, \qquad &{\rm if}\; g_2 > 0\\
0,           \; \;  \qquad     &{\rm if} \; h=g_2=0 \\
+\infty,        \quad      &{\rm if} \;\; h \neq 0, \; g_2=0.
\end{cases}
\end{align}
To see the equivalence of~\ref{sys1} and~\ref{sys2}, notice that if $g_2 > 0$, then the two systems clearly coincide. Hence, let $g_2 = 0$.
As all inequalities in second and third lines of both systems are identical, in the following, we focus on the remaining inequalities in the first lines.
Two cases arise:
\begin{itemize}
\item [(i)] $h = 0$: in this case~\ref{sys1} simplifies to $g_1 \geq 0$ and $g_2 = 0$, and by~\eqref{defcl},~\ref{sys2} simplifies to $g_1 \geq 0$ and $g_2 = 0$ as well.
\item [(ii)] $h \neq 0$:  in this case, the first inequality in~\ref{sys1} simplifies to $0 \geq h^2$, which together with $h \neq 0$ implies that the system is infeasible. By~\eqref{defcl}, the first inequality in~\ref{sys2} simplifies to $g_1 \geq +\infty$. It can be checked that $g_1$ is a concave function and its maximum value equals $X_{11}- \frac{(x_1)^2}{z_1}$. Hence in this case,~\ref{sys2} is infeasible as well.
\end{itemize}
Henceforth, we work with~\ref{sys2}; as before, for notational simplicity, whenever we write $\frac{h^2}{g_2}$, we mean ${\rm cl}(\frac{h^2}{g_2})$ as defined by~\eqref{defcl}.
It then follows that to characterize $\overline{\conv}(\S_2)$, it suffices to project out $(\tilde x^4_1, \tilde x^4_2,\lambda_4)$ from the set of points $(x,X,z,\tilde x^4_1, \tilde x^4_2,\lambda_4)$ satisfying:
\begin{align}
&X_{11} \geq \frac{(\tilde x^4_1)^2}{\lambda_4}+\frac{(x_1-\tilde x^4_1)^2}{z_1-\lambda_4}
+\frac{\Big(X_{12}-\frac{\tilde x^4_1 \tilde x^4_2}{\lambda_4 }\Big)^2}{X_{22}-\frac{(\tilde x^4_2)^2}{\lambda_4 }-\frac{(x_2 - \tilde x^4_2)^2}{(z_2-\lambda_4)}},\nonumber\\
& X_{22}- \frac{(\tilde x^4_2)^2}{\lambda_4}-\frac{(x_2 - \tilde x^4_2)^2}{z_2-\lambda_4} \geq 0, \label{workSet2}\\
& z_1 + z_2-1 \leq \lambda_4 \leq \min\{z_1, z_2\}, \quad \lambda_4 > 0\nonumber\\
&0 \leq \tilde x^4_1 \leq x_1, \quad 0 \leq \tilde x^4_2 \leq x_2,\nonumber
\end{align}
where we further assume that $(x, X,z)\in (\U_1\cup \U_2) \cap \bar \C$ as defined by~\eqref{simple},~\eqref{r1}, and~\eqref{r2}.
To perform this projection, we would like to apply Lemma~\ref{projeq} with
$$
\G:=(\U_1\cup \U_2) \cap \bar \C,
$$
and
\begin{equation*}
\begin{split}
\D:=\Big\{(x,X,z,\tilde x^4_1,& \tilde x^4_2,\lambda_4)\, : \,X_{22}- \frac{(\tilde x^4_2)^2}{\lambda_4}-\frac{(x_2 - \tilde x^4_2)^2}{z_2-\lambda_4} \geq 0, \\
&z_1 + z_2-1 \leq \lambda_4 \leq \min\{z_1, z_2\}, \,\lambda_4 > 0, \, 0 \leq \tilde x^4_1 \leq x_1, \, 0 \leq \tilde x^4_2 \leq x_2\Big\},
\end{split}
\end{equation*}
to the following optimization problem:
\begin{align}
\label{workSet3}
\tag{P$_{x,X,z}$}
\min_{\tilde x^4_1,\tilde x^4_2, \lambda_4} \quad & \frac{(\tilde x^4_1)^2}{\lambda_4}+\frac{(x_1-\tilde x^4_1)^2}{z_1-\lambda_4}
+\frac{\Big(X_{12}-\frac{\tilde x^4_1 \tilde x^4_2}{\lambda_4 }\Big)^2}{X_{22}-\frac{(\tilde x^4_2)^2}{\lambda_4 }-\frac{(x_2 - \tilde x^4_2)^2}{z_2-\lambda_4}},\\
\st \quad & X_{22}- \frac{(\tilde x^4_2)^2}{\lambda_4}-\frac{(x_2 - \tilde x^4_2)^2}{z_2-\lambda_4} \geq 0, \label{c1}\\
& z_1 + z_2-1 \leq \lambda_4 \leq \min\{z_1, z_2\}, \quad \lambda_4 > 0\label{c2}\\
&0 \leq \tilde x^4_1 \leq x_1, \quad 0 \leq \tilde x^4_2 \leq x_2. \label{c3}
\end{align}
In the following, we characterize a minimum of Problem \eqref{workSet3} for every $(x,X,z)\in (\U_1 \cup \U_2) \cap \bar \C$. This in turn, by Lemma \ref{projeq}, enables us to characterize $\overline{\conv}(\S_2)$. 
%{\color{red} We don't find for all $(x,X,z)\in \U_1 \cup \U_2$ for a subset that contains the convex hull, should fix this.}

Problem~\eqref{workSet3} is a convex optimization problem. To see this, first recall that $\tilde \Sigma$, \ie the set of points defined by System~\eqref{workSet2} is a convex set. Let us denote by $\Sigma_{x,X,z}$ the convex set defined as the restriction of $\tilde \Sigma$ to a fixed $(x_1, x_2,X_{12}, X_{22},z_1, z_2)$; then $\Sigma_{x,X,z}$ can be written as
$$
\Sigma_{x,X,z} =\Big\{(\tilde x^4_1, \tilde x^4_2, \lambda_4, X_{11}): X_{11} \geq f_{x,X,z}(\tilde x^4_1, \tilde x^4_2, \lambda_4), (\tilde x^4_1, \tilde x^4_2, \lambda_4) \in \Q_{x,X,z}\Big\},$$
where we define
\begin{equation}\label{xxx}
f_{x,X,z} := \frac{(\tilde x^4_1)^2}{\lambda_4}+\frac{(x_1-\tilde x^4_1)^2}{z_1-\lambda_4}
+\frac{\Big(X_{12}-\frac{\tilde x^4_1 \tilde x^4_2}{\lambda_4 }\Big)^2}{X_{22}-\frac{(\tilde x^4_2)^2}{\lambda_4 }-\frac{(x_2 - \tilde x^4_2)^2}{z_2-\lambda_4}}
\end{equation}
and
\begin{align*}
\Q_{x,X,z} := \Big\{&(\tilde x^4_1, \tilde x^4_2, \lambda_4): X_{22}- \frac{(\tilde x^4_2)^2}{\lambda_4}-\frac{(x_2 - \tilde x^4_2)^2}{z_2-\lambda_4} \geq 0, \\
& z_1 + z_2-1 \leq \lambda_4 \leq \min\{z_1, z_2\}, \; \lambda_4 > 0, \; 0 \leq \tilde x^4_1 \leq x_1, \; 0 \leq \tilde x^4_2 \leq x_2\Big\}.
\end{align*}
In the rest of the paper, for the sake of exposition, we will drop the dependance of $f_{x,X,z}$ on $(x,X,z)$ and we will simply write $f$.
It is simple to check that $\Q_{x,X,z}$ is a convex set. It then follows that the convex set $\Sigma_{x,X,z}$ can be described as the epigraph of function $f$ over convex set
$\Q_{x,X,z}$. Therefore, $f$ is a convex function over $\Q_{x,X,z}$; $f$ is precisely the objective function of Problem~\eqref{workSet3} while $\Q_{x,X,z}$ is its feasible set. Hence, Problem~\eqref{workSet3} is a convex optimization problem. To characterize optimal solutions of Problem~\eqref{workSet3}, we make use of the following
relaxation of Problem~\eqref{workSet3}:
\begin{align}
\label{workSet4}
\tag{P$'_{x,X,z}$}
\min_{\tilde x^4_1,\tilde x^4_2, \lambda_4} \quad & \frac{(\tilde x^4_1)^2}{\lambda_4}+\frac{(x_1-\tilde x^4_1)^2}{z_1-\lambda_4}
+\frac{\Big(X_{12}-\frac{\tilde x^4_1 \tilde x^4_2}{\lambda_4 }\Big)^2}{X_{22}-\frac{(\tilde x^4_2)^2}{\lambda_4 }-\frac{(x_2 - \tilde x^4_2)^2}{z_2-\lambda_4}},\\
\st \quad & X_{22}- \frac{(\tilde x^4_2)^2}{\lambda_4}-\frac{(x_2 - \tilde x^4_2)^2}{z_2-\lambda_4} \geq 0, \label{d1}\\
& 0 < \lambda_4 \leq \min\{z_1, z_2\}\label{d2}.
\end{align}
Problem~\eqref{workSet4} is a convex optimization problem as well. The proof follows from a similar line of arguments to that for Problem~\eqref{workSet3} by replacing $\tilde \Sigma$ with $\tilde \Sigma' :=\{(x, X, z, \tilde x^4_1, \tilde x^4_2, \lambda_4): (x, X, z, \tilde x, \tilde X, \lambda) \in  \conv(\P'_2 \cup \P'_3 \cup \P'_4)\}$, where
\begin{align*}
& \P'_2 := \{(x,X,z): z_1 = 1, z_2 = 0, \; X_{11}= x^2_1, \; x_2=X_{12}=X_{22}=0,\; x_1 \in \R\}, \\
& \P'_3 := \{(x,X,z): z_1 = 0, z_2 = 1, \; x_1=X_{11}= X_{12} =0, \; X_{22}=x^2_2, \; x_2 \in \R\},  \\
& \P'_4 := \{(x,X,z): z_1 = z_2 = 1, \; X_{11}= x^2_1, \; X_{12} =x_1 x_2, \; X_{22}=x^2_2, \; x_1, x_2 \in \R\}.
\end{align*}
However, we should remark that constraints~\eqref{d1} and~\eqref{d2} are essential for the convexity of Problems~\eqref{workSet3} and~\eqref{workSet4}.

%{\color{red}
%\begin{theorem}[Theorem~6.5 in~\cite{rf70}]
%Let $\C_i$ be a convex set in $\R^n$ for $i \in I$ (an index set). Suppose that the sets ${\rm ri}(\C)$
%have at least one point in common. Then
%$$
%\cl\Big(\bigcap_{i\in I}{\C_i}\Big) = \bigcap_{i\in I}{\cl(\C_i)}.
%$$
%\end{theorem}
%}

In the remainder of the proof, we solve Problem~\eqref{workSet3} analytically for all $(x, X, z) \in (\U_1 \cup \U_2) \cap \bar \C$.
To this end, let us we first consider Problem~\eqref{workSet4}; for any positive $x_1, X_{12}, z_1$, it can be checked that $f \geq \frac{(x_1)^2}{z_1}$ and this lower bound is attained when:
\begin{equation}\label{nec1}
\tilde x^4_1 = \lambda_4 \frac{x_1}{z_1}, \quad \tilde x^4_2 = \frac{X_{12} z_1}{x_1},
\end{equation}
where in addition $\lambda_4$ must satisfy constraints~\eqref{d1} and~\eqref{d2}. To determine $\lambda_4$, we consider the following auxiliary optimization problem:
\begin{align}\label{auxp}
&\max \quad X_{22}-\Big(\frac{X_{12} z_1}{x_1}\Big)^2 \frac{1}{\lambda_4}-\Big(x_2-\frac{X_{12} z_1}{x_1}\Big)^2\frac{1}{z_2 - \lambda_4}\\
& \st \quad 0 < \lambda_4 \leq \min\{z_1, z_2\}.\nonumber
\end{align}
If the optimal value of Problem~\eqref{auxp} is nonnegative, then its maximizer $\lambda_4$ together with $\tilde x^4_1, \tilde x^4_2$ defined by~\eqref{nec1} are an optimal solution of Problem~\eqref{workSet4}.
The objective function of Problem~\eqref{auxp} is concave over $(0, z_2)$. Setting the derivative of the objective function to zero, it follows that the maximizer of the objective function over $(0, z_2)$ is attained at:
\begin{align}
\bar \lambda_4 =
\begin{cases}
\frac{X_{12} z_1 z_2}{x_1 x_2}, \quad &{\rm if} \; X_{12} z_1 < x_1 x_2\\
\frac{X_{12} z_1 z_2}{2X_{12}z_1-x_1 x_2}, \quad & {\rm if} \; X_{12} z_1 > x_1 x_2.
\end{cases}\label{ls}
\end{align}
In the special case where $X_{12} z_1 = x_1 x_2$, Problem~\eqref{auxp} simplifies to the problem of maximizing $X_{22}-\frac{x^2_2}{\lambda_4}$ over $0 < \lambda_4 \leq \min\{z_1, z_2\}$, whose optimal value is attained at $\lambda^{*}_4 = \min\{z_1, z_2\}$.

Now let us consider Problem~\eqref{workSet3} and examine cases under which an optimal solution of this problem is given by~\eqref{nec1} and~\eqref{ls}.
%First consider the case with $z_1 = 1$. By constraints~\eqref{c2}, we get $\lambda_4 = z_2$, which in turn implies $\tilde x^4_2 = x_2$ which does not coincide with~\eqref{nec1} unless $X_{12} = x_1 x_2$; this latter region however, satisfies condition~\eqref{r0} with the convex hull given by~\eqref{peice0}. Hence in the following, we assume $z_1 < 1$.
To proceed further, we should consider the two regions $\U_1$ and $\U_2$ defined by~\eqref{r1} and~\eqref{r2} separately.
\begin{itemize}[leftmargin=*]
\item [(i)] $\U_1$: consider the case $z_2 \geq z_1$. Then we have $X_{12} z_2 > x_1 x_2$. Suppose that $X_{12} z_1 \leq x_1 x_2$.
If $X_{12} z_1 = x_1 x_2$, then $\lambda^{*}_4 = z_1$.
Now let $X_{12} z_1 < x_1 x_2$, implying the objective function of~\eqref{auxp} is maximized at $\bar\lambda_4 = \frac{X_{12} z_1 z_2}{x_1 x_2}$.
However, in this case we have $\bar\lambda_4 > z_1$, implying that a maximizer of Problem~\eqref{auxp} is $\lambda_4^* = z_1$.
Substituting $\lambda_4^*$ in~\eqref{d1}, it follows that if
$$
x_1^2 (z_2-z_1)(X_{22} z_2 - x^2_2) \geq z_1 (X_{12} z_2-x_1 x_2)^2,
$$
then an optimal solution of Problem~\eqref{workSet4} is given by
\begin{equation}\label{opt1}
\tilde x^4_1 = x_1, \quad \tilde x^4_2 = \frac{X_{12} z_1}{x_1}, \quad \lambda_4 = z_1.
\end{equation}
Since by assumption $X_{12} z_1 \leq x_1 x_2$, we conclude that~\eqref{opt1} also satisfies constraints~\eqref{c3} and hence is an optimal solution of Problem~\eqref{workSet3}
for all $(x,X,z) \in \Ri_2$, where
\begin{align}\label{reg1}
\Ri_2:=\Big\{&(x,X,z): \;z_1 \leq z_2, \; X_{12} z_2 > x_1 x_2, \; X_{12} z_1 \leq x_1 x_2, \nonumber\\
                        & x^2_1 (z_2-z_1)(X_{22} z_2 - x^2_2) \geq z_1 (X_{12} z_2-x_1 x_2)^2, X_{12} >0, \; z_1 > 0\Big\}.
\end{align}
Therefore, by Lemma~\ref{projeq}, we get
$$
\tilde \S \cap \Ri_2=\Big\{(x,X,z): X_{11} \geq \frac{x^2_1}{z_1}, \;  x_1 > 0, \; z_2 \leq 1\Big\} \cap \Ri_2.
$$
We claim that
\begin{align}\label{migor}
\cl(\tilde \S \cap \Ri_2)
% &=\cl(\{(x,X,z): X_{11} \geq \frac{x^2_1}{z_1}, \;  x_1 > 0, x_2 > 0, \; X_{12} > 0, \; z_1 > 0, \; z_2 \leq 1, \; z_1 \leq z_2, \\
%& \qquad X_{12} z_2 > x_1 x_2, \;  X_{12} z_1 \leq x_1 x_2, \; x^2_1 (z_2-z_1)(X_{22} z_2 - x^2_2) \geq z_1 (X_{12} z_2-x_1 x_2)^2 \})\\
= \chi:=\Big\{&(x,X,z): X_{11} \geq \frac{x^2_1}{z_1}, \; X_{22} \geq \frac{x^2_2}{z_2}, \; x_1\geq 0, \; z_2 \leq 1\Big\}\cap \cl(\Ri_2).
\end{align}
where
$\cl(\Ri_2)=\{(x,X,z): \; z_1 \leq z_2, \; X_{12} z_2 \geq x_1 x_2, \; X_{12} z_1 \leq x_1 x_2, \; x^2_1 (z_2-z_1)(X_{22} z_2 - x^2_2) \geq z_1 (X_{12} z_2-x_1 x_2)^2, \; X_{12} \geq 0, \; z_1 \geq 0\}$.
To show the validity of~\eqref{migor}, first note that
$\chi$ is a closed superset of $\tilde \S \cap \Ri_2$.  Moreover, for every $(x,X,z)\in \chi$, the sequence $(x^n,X^n,z^n)$:
\begin{align*}
&x^n_1=\max\Big\{x_1,\frac{1}{n}\Big\},\; x^n_2=\max\Big\{x_2,\frac{1}{n^2}\Big\},\; X^n_{12}=\max\Big\{X_{12},\frac{1}{n^2}\Big\},\; X^n_{11}=\max\Big\{X_{11}, \frac{1}{n}\Big\},\\
&X^n_{22}=\max\Big\{X_{22}, \frac{1}{n}\Big\},\; z^n_1=\max\Big\{z_1,\frac{1}{n^6}\Big\},\; z^n_2=\max\Big\{z_2,\frac{2}{n}\Big\},
\end{align*}
converges to $(x,X,z)$ as $n \to \infty$ and clearly $\{(x^n,X^n,z^n)\}\subset \tilde \S \cap \Ri_2$. Hence $\cl(\tilde \S \cap \Ri_2) = \chi$.
\vspace{0.1in}

\emph{This concludes the proof of Part~(I) of Theorem~\ref{theTh} for $i=2$.}
\vspace{0.1in}

As we show in Sections~\ref{bx1}--\ref{bx3}, for any $(x,X,z)\in \U_1 \setminus \Ri_2$, there exists an optimal solution of Problem~\eqref{workSet3} with $\tilde x^4_1 = x_1$ or with $\tilde x^4_2 = x_2$.
\vspace{0.1in}

\item [(ii)] $\U_2$: we have $X_{12} z_1 < x_1 x_2 \frac{z_1+z_2-1}{z_2} \leq x_1 x_2$.  Hence, the objective function of~\eqref{auxp} is maximized at $\bar\lambda_4 = \frac{X_{12} z_1 z_2}{x_1 x_2}$ and in this region we have $0 < \bar\lambda_4 < \min\{z_1, z_2\}$. Hence the maximum of Problem~\eqref{auxp} is attained at $\bar \lambda_4$ and is equal to $X_{22} - \frac{x^2_2}{z_2}$. Since for any $(x,X,z) \in \bar \C$ we have $X_{22} - \frac{x^2_2}{z_2} \geq 0$, it follows that $\bar \lambda_4$ together with $\tilde x^4_1, \tilde x^4_2$ defined by~\eqref{nec1} form an optimal solution of Problem~\eqref{workSet4}. Now let us consider Problem~\eqref{workSet3}; for every $(x,X,z)\in \U_2$ it holds $\tilde x^4_2 = \frac{X_{12} z_1}{x_1} < x_2$;
    however, we have $\bar \lambda_4 < z_1 + z_2 -1$. Hence, by Lemma~\ref{convbnd}, for every $(x,X,z)\in \U_2 \cap \bar \C$ there exists an optimal solution of Problem~\eqref{workSet3} with $\lambda_4 = z_1 + z_2 -1$. We analyze this boundary in Section~\ref{proof4}.

\end{itemize}

%\subsection{Region~(1) boundaries}
%\label{proof3}
%
%
%For notational simplicity, in the remainder of this paper, whenever we write a function of the form $f(u,v,w)= \frac{u v}{w}$, $u,v \geq 0$, $w > 0$, we imply its closure as defined by~\eqref{closure2}.
%

\subsection{Region~$\U_1$ boundaries: $\tilde x^4_1 = x_1$:}
\label{bx1}
Since Problem~\eqref{workSet3} is a convex optimization problem, if for some $x_1 > 0$ a feasible point $(\tilde x^4_1, \tilde x^4_2, \lambda_4)$ with
$\tilde x^4_1 = x_1$, satisfies the following conditions, then it is an optimal solution:
\begin{enumerate}[leftmargin=*]
\item $\frac{\partial f}{\partial \tilde x^4_1} \leq 0$: this condition is satisfied if
\begin{equation}\label{b1e1}
x_1 \leq \tilde x^4_2 \frac{h}{g_2},
\end{equation}
\item $\frac{\partial f}{\partial \tilde x^4_2} = 0$: this condition is satisfied if
\begin{equation}\label{b1e2}
\Big(\frac{\tilde x^4_2}{\lambda_4}-\frac{x_2-\tilde x^4_2}{z_2-\lambda_4} \Big) h^2 = \frac{x_1}{\lambda_4}h g_2,
\end{equation}

\item $\frac{\partial f}{\partial \lambda_4} = 0$: this condition is satisfied if
\begin{equation}\label{b1e3}
\Big(\frac{x_1}{\lambda_4}\Big)^2 g_2^2= \frac{2 x_1 \tilde x^4_2}{\lambda^2_4} h g_2 - \Big(\Big(\frac{\tilde x^4_2}{\lambda_4}\Big)^2-\Big(\frac{x_2-\tilde x^4_2}{z_2-\lambda_4}\Big)^2 \Big) h^2,
\end{equation}
\end{enumerate}
where $f$ is defined by~\eqref{xxx} and $g_2$, $h$ are defined by~\eqref{yyy}.
We should remark that conditions~\eqref{b1e2} and~\eqref{b1e3} are defined for $\lambda_4 \in (0, z_2)$.
%In the following, we assume that $g_2 \neq 0$, which in turn implies $h \neq 0$.
By~\eqref{b1e1}, since $x_1 >0$ and $g_2 \geq 0$, we must have $g_2 > 0$, $h > 0$, and $\tilde x^4_2 > 0$.
First we show that equation~\eqref{b1e3} is implied by equation~\eqref{b1e2}.
Using~\eqref{b1e2} and $g_2 \neq 0$, equation~\eqref{b1e3} can be written as:
\begin{align*}
\Big(\frac{x_1}{\lambda_4}\Big)^2 g_2& = \frac{2 x_1 \tilde x^4_2}{\lambda^2_4} h - \frac{x_1}{\lambda_4}\Big(\frac{\tilde x^4_2}{\lambda_4}+\frac{x_2-\tilde x^4_2}{z_2-\lambda_4} \Big) h = \frac{x_1}{\lambda_4}  \Big(\frac{\tilde x^4_2}{\lambda_4}-\frac{x_2-\tilde x^4_2}{z_2-\lambda_4}\Big) h = \frac{x_1}{\lambda_4}\frac{x_1}{\lambda_4}g_2.
\end{align*}
Hence, equation~\eqref{b1e3} is implied by equation~\eqref{b1e2}.
By feasibility we must have $g_2 > 0$; using equation~\eqref{b1e2}, inequality $g_2 > 0$ holds if and only if
$\frac{\tilde x^4_2}{\lambda_4}-\frac{x_2-\tilde x^4_2}{z_2-\lambda_4} > 0$ and $h = X_{12} - \frac{x_1\tilde x^4_2}{\lambda_4} > 0$; \ie
\begin{equation}\label{b1e4}
\lambda_4 \frac{x_2}{z_2} < \tilde x^4_2 < \lambda_4 \frac{X_{12}}{x_1}.
\end{equation}
Next, we examine the validity of inequality~\eqref{b1e1}. By equation~\eqref{b1e2} and inequality $h > 0$, inequality~\eqref{b1e1} can be equivalently written as
$\lambda_4(\frac{\tilde x^4_2}{\lambda_4}-\frac{x_2-\tilde x^4_2}{z_2-\lambda_4})  \leq \tilde x^4_2$, which using $\lambda_4 > 0$ simplifies to $\frac{x_2-\tilde x^4_2}{z_2-\lambda_4} \geq 0$, whose validity follows from inequalities~\eqref{c2} and~\eqref{c3}. Finally, we consider equation~\eqref{b1e2}; first note that
for $\lambda_4 \in (0, z_2)$, this equation is equivalent to
$x_1 (X_{22} z_2-x^2_2)+\lambda_4 (x_2 X_{12} - x_1 X_{22})  = \tilde x^4_2 (X_{12} z_2-x_1 x_2) $; moreover, by inequality~\eqref{b1e4} we have
$X_{12} z_2-x_1 x_2 > 0$; it then follows that:
\begin{equation}\label{b1e5}
\tilde x^4_2 = \frac{x_1 (X_{22} z_2-x^2_2)+\lambda_4 (x_2 X_{12} - x_1 X_{22})}{X_{12} z_2-x_1 x_2}.
\end{equation}
Hence if $(x_1, \tilde x^4_2, \lambda_4)$ satisfies~\eqref{c2},~\eqref{c3},~\eqref{b1e4} and~\eqref{b1e5}, then it is an optimal solution of Problem~\eqref{workSet3}. Substituting~\eqref{b1e5} into~\eqref{b1e4}, it follows that $\lambda_4 \frac{x_2}{z_2} < \tilde x^4_2 $ simplifies to $\lambda_4 < z_2$ which holds by~\eqref{c2},
and $\tilde x^4_2 < \lambda_4 \frac{X_{12}}{x_1}$
 is equivalent to
\begin{equation}\label{b1e7}
(X_{12} z_2 - x_1 x_2)^2 > \Big(\frac{z_2}{\lambda_4}-1\Big)x^2_1(X_{22} z_2 - x^2_2).
\end{equation}
Next, we examine constraints~\eqref{c3}; using~\eqref{b1e5} it follows that
$\tilde x^4_2 \leq x_2$ is equivalent to
\begin{equation}\label{b1e6}
X_{12} x_2 \geq X_{22}x_1,
\end{equation}
which together with the fact that $X_{22} z_2 \geq x^2_2$ for all $(x,X,z) \in \bar \C$ implies $\tilde x^4_2 \geq 0$. To determine
the optimal $\lambda_4$, notice that the right-hand side of inequality~\eqref{b1e7} is decreasing in $\lambda_4$. Hence we should set $\lambda_4$ at its upper bound in the interval defined by \eqref{c2}.
Two cases arise:
\begin{itemize}[leftmargin=*]
\item If $z_1 < z_2$, then $\lambda_4 = z_1$; substituting in~\eqref{b1e7}, we conclude that for every $(x,X,z) \in \Ri_3$, where
\begin{align}\label{reg2}
\Ri_3:=\Big\{& (x,X,z): \; z_1 < z_2, \; X_{12} x_2 > X_{22} x_1, \;  x_1 \geq 0,\; X_{12} > 0, \; z_1 >0, \nonumber\\
& z_1(X_{12} z_2 - x_1 x_2)^2 > x^2_1(z_2-z_1)(X_{22} z_2 - x^2_2)\Big\},
\end{align}
an optimal solution of Problem~\eqref{workSet3} is given by
$$
(\tilde x^4_1, \tilde x^4_2, \lambda_4) = \Big(x_1, \; \frac{X_{12} x_2 z_1 + x_1 (X_{22} (z_2 - z_1)-x_2^2)}{X_{12} z_2-x_1 x_2}, \; z_1\Big).
$$
Therefore, by Lemma~\ref{projeq}, we obtain
%over region $\Ri_3$, the projection of $\tilde \Sigma$ defined by~\eqref{workSet} onto the space $(x,X,z)$ is given by
\begin{align*}
\tilde \S \cap \Ri_3=\Big\{(x,X,z): \; X_{22} \geq \frac{x^2_2}{z_2}, \; \Big(X_{11} - \frac{x^2_1}{z_2}\Big)\Big(X_{22}- \frac{x^2_2}{z_2}\Big) \geq \Big(X_{12}- \frac{x_1 x_2}{z_2}\Big)^2, z_2 \leq 1\Big\} \cap \Ri_3.
\end{align*}
%By Theorem~\ref{intcl}, we have $\overline{\conv}(\S) \cap \cl(\Ri_3)=\cl(\tilde \S \cap \Ri_3)$.
Moreover, employing a similar line of arguments to that used to prove~\eqref{migor}, we deduce that
\begin{align}\label{peice1}
 &\cl(\tilde\S \cap \Ri_3)  =\nonumber\\
 & \Big\{(x,X,z): \; X_{22} \geq \frac{x^2_2}{z_2}, \; \Big(X_{11} - \frac{x^2_1}{z_2}\Big)\Big(X_{22}- \frac{x^2_2}{z_2}\Big) \geq \Big(X_{12}- \frac{x_1 x_2}{z_2}\Big)^2, \; z_2 \leq 1\Big\} \cap \cl(\Ri_3),
\end{align}
where $\cl(\Ri_3) = \{(x,X,z): \; z_1 \leq z_2, \; X_{12} x_2 \geq X_{22} x_1, \;  x_1 \geq 0, \; X_{12} \geq 0, \; z_1 \geq 0, \; z_1(X_{12} z_2 - x_1 x_2)^2 \geq x^2_1(z_2-z_1)(X_{22} z_2 - x^2_2)\}$.
%Indeed
%\begin{align*}
%&\cl( \Big\{\Big(X_{11} - \frac{x^2_1}{z_2}\Big)\Big(X_{22}- \frac{x^2_2}{z_2}\Big) \geq \Big(X_{12}- \frac{x_1 x_2}{z_2}\Big)^2, \; X_{22} \geq \frac{x^2_2}{z_2}, \; x_1 > 0,\;  x_2 > 0, \; X_{12} > 0, \; z_1 > 0, \; z_2 \leq 1\Big\}\\
%&\quad \cap \{ z_1 < z_2, \quad X_{12} x_2 \geq X_{22} x_1, \quad  z_1(X_{12} z_2 - x_1 x_2)^2 \geq x^2_1(z_2-z_1)(X_{22} z_2 - x^2_2)\})\\
%&= \Big\{\Big(X_{11} - \frac{x^2_1}{z_2}\Big)\Big(X_{22}- \frac{x^2_2}{z_2}\Big) \geq \Big(X_{12}- \frac{x_1 x_2}{z_2}\Big)^2,\; X_{22} \geq \frac{x^2_2}{z_2}, \; x_1 \geq 0,\;
% x_2 \geq 0, \; X_{12} \geq 0, \; z_1 \geq 0, \; z_2 \leq 1\Big\}\\
%&\quad \cap \{ z_1 \leq z_2, \quad X_{12} x_2 \geq X_{22} x_1, \quad  z_1(X_{12} z_2 - x_1 x_2)^2 \geq x^2_1(z_2-z_1)(X_{22} z_2 - x^2_2)\}=:A.
%\end{align*}
%
%We observe that
%\begin{align*}
%\cl &\Big\{(x,X,z): X_{11}-\frac{x_1^2}{z_1} \geq 0, \;
%X_{22}- \frac{x_2^2}{z_2}\geq 0, \; X_{12} > 0, \; x_1 > 0, \; x_2 > 0, \; z_1 > 0, \; z_2 > 0\}\\
%=A:=& \Big\{(x,X,z): X_{11}-\frac{x_1^2}{z_1} \geq 0, \;
%X_{22}- \frac{x_2^2}{z_2}\geq 0, \; X_{12} \geq 0, \; x_1 \geq 0, \; x_2 \geq 0, \; z_1 \geq 0, \; z_2 \geq 0\}.
%\end{align*}
\vspace{0.1in}

\emph{This concludes the proof of Part~(II) of Theorem~\ref{theTh} for $i=3$.}

%
%{\color{blue}
%We observe that
%\begin{align*}
%\cl & \Big\{\Big(X_{11} - \frac{x^2_1}{z_2}\Big)\Big(X_{22}- \frac{x^2_2}{z_2}\Big) \geq \Big(X_{12}- \frac{x_1 x_2}{z_2}\Big)^2.\nonumber\\
%& X_{22} \geq \frac{x^2_2}{z_2}, \; x_1 > 0,\;
% x_2 > 0, \; X_{12} > 0, \; z_1 > 0, \; z_2 \leq 1\Big\}\\
%=A:=& \Big\{\Big(X_{11} - \frac{x^2_1}{z_2}\Big)\Big(X_{22}- \frac{x^2_2}{z_2}\Big) \geq \Big(X_{12}- \frac{x_1 x_2}{z_2}\Big)^2.\nonumber\\
%& X_{22} \geq \frac{x^2_2}{z_2}, \; x_1 \geq 0,\;
% x_2 \geq 0, \; X_{12} \geq 0, \; z_1 \geq 0, \; z_2 \leq 1\Big\}.
%\end{align*}
%Indeed, for every $(x,X,z)\in A$ the sequence $(x^n,X^n,z^n)$ defined as $x^n_1=x_1+1/n$, $x^n_2=x_2+1/n$, $X^n_{12}=X^n_{12}+1/n$, $X^n_{11}=X_{11}+ \frac{1}{n^2z_2}+ \frac{2x_1}{nz_2}\}$, $X^n_{22}=\max\{X_{22}, \frac{(x_2+1/n)^2}{z_2+1/n}\}$, $z^n_1=z_1+1/n$, $z^n_2=z_2+1/n$, converges to $(x,X,z)$.
%}
%
%
%{\color{red} We need to find the closure of the following intersected with~\eqref{reg2}
%\begin{align*}
% \Big\{&\Big(X_{11} - \frac{x^2_1}{z_2}\Big)\Big(X_{22}- \frac{x^2_2}{z_2}\Big) \geq \Big(X_{12}- \frac{x_1 x_2}{z_2}\Big)^2.\nonumber\\
%& X_{22} \geq \frac{x^2_2}{z_2}, \; x_1 > 0,\;
% x_2 > 0, \; X_{12} > 0, \; z_1 > 0, \; z_2 \leq 1\Big\}.
%\end{align*}
%We should give a meaning to $\frac{x_1 x_2}{z_2}$ at $z_2 = 0$.
%}
%

\item If $z_2 \leq z_1$: since equations~\eqref{b1e2} and~\eqref{b1e3} are defined over $\lambda_4 \in (0, z_2)$, we cannot set $\lambda_4=z_2$.

We define
\begin{equation}\label{reg4}
\Ri_4:=\Big\{(x,X,z): \; z_2 \leq z_1, \; X_{12} x_2 > X_{22} x_1, \; x_1 \geq 0, \; X_{12} > 0, \; z_2 >0 \Big\}.
\end{equation}
Recall that inequality~\eqref{b1e4} implies
$X_{12} z_2-x_1 x_2 > 0$. Hence, for every $(x,X,z)\in \Ri_4$, there exists $0 <\varepsilon(x,X,z) < z_2$ such that $(x,X,z)$ satisfies:
$$
(X_{12} z_2 - x_1 x_2)^2 > \frac{\varepsilon}{z_2 - \varepsilon}x^2_1(X_{22} z_2 - x^2_2).
$$
It follows that for any $(x,X,z)\in \Ri_4$, an optimal solution of Problem~\eqref{workSet3} is given by
$$
(\tilde x^4_1, \tilde x^4_2, \lambda_4) = \Big(x_1, \; x_2-\varepsilon (\frac{X_{12}x_2 -X_{22}x_1 }{X_{12}z_2 - x_1 x_2}), \; z_2-\varepsilon\Big).
$$
Therefore by Lemma~\ref{projeq}, we obtain:
\begin{align}\label{peice2}
\tilde \S \cap \Ri_4= \Big\{(x,X,z): \; & X_{22} \geq \frac{x^2_2}{z_2}, \;  \Big(X_{11}- \frac{x^2_1}{z_2}\Big)\Big(X_{22} - \frac{x^2_2}{z_2}\Big) \geq \Big(X_{12} - \frac{x_1 x_2}{z_2}\Big)^2,\; z_1 \leq 1\Big\} \cap \Ri_4.
\end{align}
Using a similar line of arguments to those used to prove~\eqref{migor}, we obtain:
\begin{align*}
\cl(\tilde \S \cap \Ri_4)= \Big\{(x,X,z): \; &X_{22} \geq \frac{x^2_2}{z_2}, \;  \Big(X_{11}- \frac{x^2_1}{z_2}\Big)\Big(X_{22} - \frac{x^2_2}{z_2}\Big) \geq \Big(X_{12} - \frac{x_1 x_2}{z_2}\Big)^2, z_1 \leq 1\Big\}\cap \cl (\Ri_4),
\end{align*}
where $\cl (\Ri_4)= \{(x,X,z): \; z_2 \leq z_1, \; X_{12} x_2 \geq X_{22} x_1, \; x_1 \geq 0, \; X_{12} \geq 0, \; z_2 \geq 0\}$.

\vspace{0.1in}

\emph{This concludes the proof of Part~(II) of Theorem~\ref{theTh} for $i=4$.}

\end{itemize}

\subsection{Region~$\U_1$ boundaries: $\tilde x^4_2 = x_2$}
\label{bx2}
Since Problem~\eqref{workSet3} is a convex optimization problem, if for some $x_2 > 0$ a feasible point $(\tilde x^4_1, \tilde x^4_2, \lambda_4)$ with
$\tilde x^4_2 = x_2$, satisfies the following conditions, then it is an optimal solution:
\begin{enumerate}[leftmargin=*]
\item $\frac{\partial f}{\partial \tilde x^4_1} = 0$: this condition is satisfied if
\begin{equation}\label{b1e8}
\Big(\frac{\tilde x^4_1}{\lambda_4}-\frac{x_1-\tilde x^4_1}{z_1-\lambda_4} \Big) g_2 = \frac{x_2}{\lambda_4} h,
\end{equation}
\item $\frac{\partial f}{\partial \tilde x^4_2} \leq 0$: this condition is satisfied if
\begin{equation}\label{b1e9}
 \Big(X_{12}-\frac{\tilde x^4_1 x_2}{\lambda_4}\Big)\Big(X_{12} x_2 -X_{22}\tilde x^4_1\Big) \leq 0,
\end{equation}
\item $\frac{\partial f}{\partial \lambda_4} = 0$: this condition is satisfied if
\begin{equation}\label{b1e10}
\Big(\Big(\frac{\tilde x^4_1}{\lambda_4}\Big)^2-\Big(\frac{x_1-\tilde x^4_1}{z_1-\lambda_4} \Big)^2 \Big) g^2_2 = \frac{2 \tilde x^4_1 x_2}{\lambda^2_4} h g_2 -\Big(\frac{x_2}{\lambda_4}\Big)^2 h^2,
\end{equation}
\end{enumerate}
where $f$ is defined by~\eqref{xxx} and $g_2$, $h$ are defined by~\eqref{yyy}.
We should remark that conditions~\eqref{b1e8} and~\eqref{b1e10} are defined for $\lambda_4 \in (0, z_1)$.
Note that by~\eqref{b1e8}, if $h =0$, then we must have $\tilde x^4_1 = \lambda_4 \frac{x_1}{z_1} = \lambda_4 \frac{X_{12}}{x_2}$, which holds
only if $X_{12} z_1 = x_1 x_2$.
Henceforth, we assume that $h \neq 0$.
%$g_2 \neq 0$, which in turn implies $h \neq 0$.
First we show that equation~\eqref{b1e10} is implied by equation~\eqref{b1e8}. Using~\eqref{b1e8} and $h \neq 0$, equation~\eqref{b1e10}
can be written as
\begin{align*}
\Big(\Big(\frac{\tilde x^4_1}{\lambda_4}\Big)^2-\Big(\frac{x_1-\tilde x^4_1}{z_1-\lambda_4} \Big)^2 \Big) g^2_2 = &\frac{2 \tilde x^4_1 x_2}{\lambda^2_4} h g_2 -\Big(\frac{x_2}{\lambda_4}\Big)^2 h^2\\
\frac{x_2 h}{\lambda_4} \Big(\frac{\tilde x^4_1}{\lambda_4}+\frac{x_1-\tilde x^4_1}{z_1-\lambda_4} \Big) g_2
=&\frac{x_2 h}{\lambda_4}(\frac{2 \tilde x^4_1}{\lambda_4} g_2 -\frac{x_2}{\lambda_4} h)\\
\frac{x_2}{\lambda_4} h=&(\frac{\tilde x^4_1}{\lambda_4} -\frac{x_1-\tilde x^4_1}{z_1-\lambda_4}) g_2,
 \end{align*}
where the validity of the last equality follows from~\eqref{b1e8}. We now examine the validity of inequality~\eqref{b1e9}.
First assume that $X_{12}\lambda_4-\tilde x^4_1 x_2 \leq 0$ and $X_{12} x_2 -X_{22}\tilde x^4_1 \geq 0$: using the valid inequality $X_{22} z_2 \geq x^2_2$,
it follows that $X_{12} x_2 \geq X_{22}\tilde x^4_1 \geq \frac{x^2_2 \tilde x^4_1}{z_2}$, implying $X_{12} z_2 \geq \tilde x^4_1 x_2$, which
together with $\lambda_4 \leq z_2$ contradicts the inequality $X_{12}\lambda_4-\tilde x^4_1 x_2 \leq 0$. Hence this case is not possible.
Therefore, using $h \neq 0$, inequality~\eqref{b1e9} is satisfied if and only if
%$X_{12}\lambda_4-\tilde x^4_1 x_2 \geq 0$, and $X_{12} x_2 -X_{22}\tilde x^4_1 \leq 0$; \ie
%\begin{equation}\label{b1e11}
%\frac{x_2 X_{12}}{X_{22}}\leq \tilde x^4_1 \leq  \frac{\lambda_4 X_{12}}{x_2}
%\end{equation}
\begin{equation}\label{es2}
X_{22}\tilde x^4_1- X_{12} x_2 \geq 0,
\end{equation}
and
\begin{equation}\label{es1}
X_{12}\lambda_4-\tilde x^4_1 x_2 > 0.
\end{equation}
By feasibility we must have $g_2 > 0$; \ie
\begin{equation}\label{b1f11}
X_{22} - \frac{x^2_2}{\lambda_4} > 0.
\end{equation}
Next, we consider equation~\eqref{b1e8}; by~\eqref{b1f11} and $\lambda_4 \leq z_1$, we conclude that $X_{22} z_1 - x^2_2 > 0$. Then,
solving equation~\eqref{b1e8} for $\tilde x^4_1$, we obtain
\begin{equation}\label{x1sol}
\tilde x^4_1 = \frac{x_2 (X_{12}z_1 -x_1 x_2)+\lambda_4(X_{22}x_1-X_{12}x_2)}{X_{22} z_1 - x^2_2}.
\end{equation}
Substituting~\eqref{x1sol} in inequality~\eqref{es2}, it follows that this inequality is
satisfied if, together with~\eqref{b1f11}, we have
\begin{equation}\label{b1e12}
X_{22} x_1 - X_{12} x_2 \geq 0,
\end{equation}
and inequality~\eqref{es1} is
satisfied if together with~\eqref{b1f11}, we have
\begin{equation}\label{b1e13}
X_{12} z_1 -x_1 x_2 > 0.
\end{equation}
It is then simple to check that by~\eqref{x1sol},~\eqref{b1e12} and~\eqref{b1e13}, the constraint $0 \leq \tilde x^4_1 \leq x_1$ is always satisfied.
Finally, to optimally determine $\lambda_4$, notice that the left-hand side of inequality~\eqref{b1f11} is increasing in $\lambda_4$.
Hence we should set $\lambda_4$ at its upper bound in the interval defined by \eqref{c2}. Two cases arise:
\begin{itemize}[leftmargin=*]
\item If $z_2 < z_1$, then $\lambda_4 = z_2$; substituting in inequality~\eqref{b1f11} yields $X_{22} > \frac{x^2_2}{z_2}$.
%, which is valid for all $(x,X,z) \in \overline{\conv}(\S_2)$.
Hence we conclude that for any  $(x,X,z) \in \Ri'_5$, where
\begin{equation}\label{reg5}
\Ri'_5 := \Big\{(x,X,z): \; z_2 < z_1, \;  X_{12} z_1 > x_1 x_2, \; X_{22} x_1 \geq X_{12} x_2, \; x_2 \geq 0, \; X_{12} > 0,\; z_2 > 0\Big\},
\end{equation}
an optimal solution of Problem~\eqref{workSet3} is attained at
$$
(\tilde x^4_1, \; \tilde x^4_2,\; \lambda_4 )= \Big(\frac{x_1(X_{22}z_2-x^2_2)+X_{12}x_2(z_1-z_2)}{X_{22} z_1 -x^2_2}, \; x_2, \;z_2\Big).
$$
%Notice that $X_{22} z_1 -x^2_2>0$ since $X_{22} z_2 -x^2_2 \geq 0$ and $z_2 < z_1$.
Therefore by Lemma~\ref{projeq}, we get
\begin{align}\label{int1}
\tilde \S \cap \Ri'_5= \Big\{& (x,X,z): \; X_{22} > \frac{x^2_2}{z_2}, \; \Big(X_{11} - \frac{x^2_1}{z_1}\Big)\Big(X_{22} - \frac{x^2_2}{z_1}\Big) \geq \Big(X_{12}  -\frac{x_1 x_2}{z_1}\Big)^2, \;  x_1> 0, \nonumber\\
& z_1 \leq 1\Big\} \cap \Ri'_5.
\end{align}

\item $z_1 \leq z_2$: since conditions~\eqref{b1e8} and~\eqref{b1e10} are defined for $\lambda_4 \in (0, z_1)$, we cannot set $\lambda_4=z_1$.

We define
\begin{equation}\label{reg7}
\Ri''_5:=\Big\{(x,X,z): \; z_1 \leq z_2, \;  X_{12} z_1  > x_1 x_2, \; X_{22} x_1 \geq X_{12} x_2, \; x_2 \geq 0\; X_{12} > 0, \; z_1 >0\Big\}.
\end{equation}
For any $(x,X,z)\in  \Ri''_5$, by $X_{12} z_1  > x_1 x_2$ and $X_{22} x_1 \geq X_{12} x_2$, we deduce that $X_{22} > \frac{x^2_2}{z_1}$. Hence, there exists $0 <\varepsilon(x,X,z) < z_1$ such that $(x,X,z)$ satisfies
$$
 X_{22} >  \frac{x^2_2}{z_1-\varepsilon}.
$$
It follows that an optimal solution of Problem~\eqref{workSet3} is attained at
$$
(\tilde x^4_1, \;\tilde x^4_2, \; \lambda_4)= \Big(x_1 - \varepsilon \Big(\frac{X_{22}x_1-X_{12}x_2}{X_{22} z_1 -x^2_2}\Big), \; x_2, \;  z_1-\varepsilon\Big).
$$
Therefore by Lemma~\ref{projeq}, we obtain:
\begin{align}\label{int2}
 \tilde \S \cap \Ri''_5=\Big\{&(x,X,z): \; \Big(X_{11} - \frac{x^2_1}{z_1}\Big)\Big(X_{22} - \frac{x^2_2}{z_1}\Big) \geq \Big(X_{12}-\frac{x_1 x_2}{z_1}\Big)^2, x_1> 0,
  z_2 \leq 1\Big\}\cap \Ri''_5.
\end{align}
\end{itemize}

From~\eqref{reg5},~\eqref{int1},~\eqref{reg7}, and~\eqref{int2}, it follows that
\begin{align}\label{peice3}
\cl(\tilde \S \cap \Ri_5)= \Big\{& (x,X,z): \; X_{22} \geq \frac{x^2_2}{z_2}, \; \Big(X_{11} - \frac{x^2_1}{z_1}\Big)\Big(X_{22} - \frac{x^2_2}{z_1}\Big) \geq \Big(X_{12}  -\frac{x_1 x_2}{z_1}\Big)^2, \nonumber\\
& x_1 \geq 0, \; X_{12} \geq 0,\; z_1, z_2 \leq 1\Big\} \cap \cl(\Ri_5),
\end{align}
where $\Ri_5 = \Ri'_5 \cup \Ri''_5$ and $\cl(\Ri_5) = \{(x,X,z): X_{12} z_1 \geq x_1 x_2, \; X_{22} x_1 \geq X_{12} x_2, x_2 \geq 0, X_{12} \geq 0, z_1, z_2 \geq 0\}$.
\vspace{0.1in}

\emph{This concludes the proof of Part~(III) of Theorem~\ref{theTh} for $i=1$.}

%{\color{red}
%The closure arguments in this part should follow by symmetry from the previous boundary case.
%}

\subsection{Convex hull characterization over Region~$\U_1$}
\label{bx3}
We now show that the union of the regions defined by the system of inequalities~\eqref{reg1},~\eqref{reg2},~\eqref{reg4},~\eqref{reg5},
and~\eqref{reg7} covers $\U_1$; \ie $\Ri_2 \cup \Ri_3 \cup \Ri_4 \cup \Ri_5\supseteq \U_1 \cap \bar \C$.
Two cases arise:
\begin{itemize}[leftmargin=*]
\item [(i)] $z_1 \leq z_2$: in this case by~definition of $\U_1$, we have $X_{12} z_2 > x_1 x_2$.
Hence it suffices to show that any point $(x,X,z) \in \bar \C$ satisfying $X_{12} z_2 > x_1 x_2$ is also in the set $\Ri_2 \cup \Ri_3 \cup \Ri_5$.
Let us define the sets $\C_1 = \{(x,X,z): X_{12} z_1 \leq x_1 x_2\}$,
$\C_2=\{(x,X,z): X_{12} x_2 > x_1 X_{22}\}$, $\C_3= \{(x,X,z): x^2_1(z_2-z_1)(X_{22}z_2-x^2_2) \geq z_1 (X_{12}z_2-x_1 x_2)^2\}$, which can also be written as
$\C_3= \{(x,X,z): x^2_1(z_2-z_1)X_{22} \geq X^2_{12} z_1 z_2-2X_{12}x_1 x_2z_1  +(x_1 x_2)^2\}$.
%Denote by $\land$ the ``and'' operation by denote by $\lnot$ the ``negation'' operation.
Then to prove the statement, it suffices to show that the following sets are empty:
\begin{itemize}[leftmargin=*]
\item [$\bullet$] $\bar \C \cap \C_1 \cap \C^c_2 \cap \C^c_3$: by $(x,X,z) \in \C^c_2$ and $z_1 \leq z_2$, we have $x^2_1  (z_2-z_1) X_{22} \geq x_1 x_2 X_{12}(z_2-z_1)$. Moreover, since
$X_{12} z_2 -x_1 x_2 > 0$, it follows that $X_{12} x_1 x_2 (z_2-z_1) > X^2_{12} z_1 z_2 -2 x_1 x_2 X_{12} z_1 +(x_1 x_2)^2$, implying $x^2_1 (z_2-z_1) X_{22} > X^2_{12} z_1 z_2 -2 X_{12}x_1 x_2 z_1 +(x_1 x_2)^2$, which is in contradiction with $(x,X,z) \in \C^c_3$.

%\item [$\bullet$] $\C^c_1 \cap \C^c_2 \cap \C^c_4 \cap \B$: since $(x,X,z) \in \C^c_1 \cap \C^c_2$,  we get $X_{22} \geq \frac{X_{12}x_2}{x_1} > \frac{x^2_2}{z_1}$, implying $
%X_{22} z_1 >x^2_2$, which is in contradiction with $(x,X,z) \in \C^c_4$.

\item [$\bullet$] $\bar \C \cap \C^c_1 \cap \C_2 \cap \C_3$: by $(x,X,z) \in \C_2$ and $z_1 \leq z_2$, we have $x^2_1X_{22} (z_2-z_1) \leq X_{12} x_1 x_2 (z_2-z_1)$. Since $(x,X,z) \in \C^c_1$; \ie $X_{12} z_1 -x_1 x_2 > 0$
and $X_{12} z_2 - x_1 x_2 > 0$, we can multiply the left-hand side of the two inequalities to obtain
 $X_{12} x_1 x_2 (z_2-z_1) < X^2_{12} z_1 z_2 -2 x_1 x_2 X_{12} z_1 + (x_1 x_2)^2$. Combining these two inequalities we get a contradiction with $(x,X,z) \in \C_3$.

\end{itemize}

\item [(ii)] $z_2 \leq z_1$: in this case by definition of $\U_1$, we have $X_{12} z_1 > x_1 x_2$.
Hence it suffices to show that any point $(x,X,z) \in \bar \C$ satisfying $X_{12} z_1 > x_1 x_2$ is also in the set $\Ri_4 \cup \Ri_5$. To this end,
we need to show that for every point $(x,X,z) \in \bar \C$ satisfying $X_{12} z_1 > x_1 x_2$, it either does not satisfy inequality $X_{12} x_2 > x_1 X_{22}$ or it does not satisfy inequality $X_{12} z_2 \leq x_1 x_2$. Indeed, since $X_{22} z_2 \geq x^2_2$ for all $(x,X,z) \in \bar \C$, using $X_{12} x_2 > x_1 X_{22}$, we get
$X_{12} x_2 > x_1 \frac{x^2_2}{z_2}$, which is in contradiction with $X_{12} z_2 \leq x_1 x_2$.

\end{itemize}

\subsection{Region~$\U_2$ boundary: $\lambda_4 = z_1 + z_2-1$}
\label{proof4}

As we detailed in Section~\ref{proof2}, for all $(x,X,z)\in \U_2$, there always exists an optimal solution
of Problem~\eqref{workSet3} with $\lambda_4 = z_1 + z_2 -1$. Hence, to characterize $\overline{\conv}(\S_2)$ over $\U_2$, 
given any $(x,X,z)\in \U_2 \cap \bar \C$, we solve the following convex
optimization problem analytically:
\begin{align}
\label{workSet5}
\tag{$\bar {\rm P}_{x,X,z}$}
\min_{\tilde x^4_1,\tilde x^4_2} \quad & \frac{(\tilde x^4_1)^2}{z_1+z_2-1}+\frac{(x_1-\tilde x^4_1)^2}{1-z_2}
+\frac{\Big(X_{12}-\frac{\tilde x^4_1 \tilde x^4_2}{z_1+z_2-1}\Big)^2}{X_{22}-\frac{(\tilde x^4_2)^2}{z_1+z_2-1}-\frac{(x_2 - \tilde x^4_2)^2}{1-z_1}},\\
\st \quad & X_{22}- \frac{(\tilde x^4_2)^2}{z_1+z_2-1}-\frac{(x_2 - \tilde x^4_2)^2}{1-z_1} \geq 0, \label{f1}\\
&0 \leq \tilde x^4_1 \leq x_1, \quad 0 \leq \tilde x^4_2 \leq x_2. \label{f2}
\end{align}
Notice that Problem~\eqref{workSet5} is convex as it is obtained by fixing $\lambda_4=z_1+z_2-1$ in the convex Problem~\eqref{workSet3}.
By inequalities~\eqref{r2} together with $x_1, x_2 \geq 0$, we deduce that $x_1, x_2 > 0$ and $z_1 + z_2 > 1$.
%{\color{red} First let $z_1 = 1$; it then follows that $\lambda_4 = z_2$, which in turn implies $\tilde x^4_2 = x_2$. As we detailed in Section~\eqref{bx2}, in this case for optimality inequality~\eqref{findit} must be satisfied which for $z_1 = 1$ simplifies to $X_{12} \geq x_1 x_2$. However, over Region~(2)
%defined by~\eqref{r2} we have }
We start by characterizing all points of zero-gradient of the objective function of Problem~\eqref{workSet5}, which with a slight abuse of notation we will keep denoting by $f$.
Subsequently, we identify conditions under which each point of zero-gradient is feasible, which by convexity of
Problem~\eqref{workSet5} implies its optimality. It can be checked that
\begin{align}
& \frac{\partial f}{\partial \tilde x^4_1} = 0 \quad \Leftrightarrow \quad \frac{\tilde x^4_1}{z_1 + z_2-1}-\frac{x_1-\tilde x^4_1}{1-z_2} = \frac{\tilde x^4_2}{z_1 + z_2-1} \frac{h}{g_2}\label{px1},\\
& \frac{\partial f}{\partial \tilde x^4_2} = 0 \quad \Leftrightarrow \quad \frac{\tilde x^4_1}{z_1 + z_2-1} h g_2 =\Big(\frac{\tilde x^4_2}{z_1+z_2-1}-\frac{x_2-\tilde x^4_2}{1-z_1}\Big) h^2\label{px2},
\end{align}
where as before we define
$$
h = X_{12}-\frac{\tilde x^4_1 \tilde x^4_2}{z_1 + z_2-1}, \qquad g_2 = X_{22}- \frac{(\tilde x^4_2)^2}{z_1+z_2-1}-\frac{(x_2 - \tilde x^4_2)^2}{1-z_1}.
$$
Notice that conditions~\eqref{px1} and~\eqref{px2} are defined for $z_1, z_2 \in (0,1)$. Two cases arise:
\begin{itemize}[leftmargin=*]
\item $h = 0$: in this case, by~\eqref{px1} and~\eqref{px2}, the point of zero-gradient is given by:
\begin{equation}\label{opth0}
\tilde x^4_1 = (z_1+z_2-1) \frac{x_1}{z_1}, \qquad \tilde x^4_2 = X_{12} \frac{z_1}{x_1}.
\end{equation}
It can be checked that for any $(x,X,z) \in \U_2$, the above point satisfies constraints~\eqref{f2}. Hence, if inequality~\eqref{f1} is satisfied; \ie
\begin{equation}\label{reg8}
(1-z_1)(z_1+z_2-1) x^2_1 (X_{22} z_2 - x^2_2) \geq \Big(X_{12} z_1 z_2-x_1 x_2(z_1+z_2-1)\Big)^2,
\end{equation}
then $(\tilde x^4_1, \tilde x^4_2)$ given by~\eqref{opth0} is an optimal solution of Problem~\eqref{workSet5}. Let us denote by
$\Ri_6$, the region defined by inequalities~\eqref{r2} and~\eqref{reg8}. Then by Lemma~\ref{projeq}, we get:
$$
\tilde \S \cap \Ri_6 = \Big\{(x,X,z): \; X_{11} \geq \frac{x^2_1}{z_1}, \; x_1, x_2 > 0,\; z_1, z_2 < 1\Big\} \cap \Ri_6.
$$
which in turn implies that
\begin{align}\label{peice5}
\cl(\tilde \S\cap \Ri_6)= &\Big\{(x,X,z): \; X_{11} \geq \frac{x^2_1}{z_1}, \; X_{22} \geq \frac{x^2_2}{z_2}, \; x_1, x_2 \geq 0,\;z_1, z_2 \leq 1\Big\}\cap \cl(\Ri_6),
\end{align}
where $\cl(\Ri_6)$ consists of the set of points satisfying inequality~\eqref{reg8} together with inequalities $X_{12} z_1 z_2 \leq x_1 x_2 (z_1 + z_2 -1)$, $X_{12} \geq 0$,
$z_1, z_2 \geq 0$.
\vspace{0.1in}

\emph{This concludes the proof of Part~(I) of Theorem~\ref{theTh} for $i=6$.}

%{\color{red}
%Need to take the closure of the set  $X_{11} \geq \frac{x^2_1}{z_1}, \; x_1 > 0, x_2 > 0, X_{12} >0, \; z_1, z_2 \in (0,1)$ intersected with inequalities~\eqref{r2} and~\eqref{reg3}.
%}

\item $h \neq 0$: in this case at any minimizer we must have $g_2 \neq 0$; then by~\eqref{px1} and~\eqref{px2}, a point of zero-gradient satisfies the following:
\begin{align}
& \frac{\tilde x^4_1}{z_1 + z_2-1}-\frac{x_1-\tilde x^4_1}{1-z_2} = \frac{\tilde x^4_2}{z_1 + z_2-1} \frac{h}{g_2}\label{ppx1},\\
& \frac{\tilde x^4_1}{z_1 + z_2-1} =\Big(\frac{\tilde x^4_2}{z_1+z_2-1}-\frac{x_2-\tilde x^4_2}{1-z_1}\Big) \frac{h}{g_2}\label{ppx2}.
\end{align}
From~\eqref{ppx1} and~\eqref{ppx2}, it follows that
$$
\frac{\tilde x^4_1}{z_1 + z_2-1} \frac{(x_2-\tilde x^4_2)}{1-z_1}+\frac{\tilde x^4_2}{z_1+z_2-1}\frac{(x_1-\tilde x^4_1)}{1-z_2}
=\frac{(x_1-\tilde x^4_1)(x_2-\tilde x^4_2)}{(1-z_1)(1-z_2)},
$$
which (since $\tilde x^4_2=x_2 z_1$ does not satisfy the equation) can be equivalently written as:
\begin{equation}\label{use}
\tilde x^4_1= x_1 \Big(\frac{\tilde x^4_2 z_2 -x_2(z_1+z_2-1)}{\tilde x^4_2-x_2 z_1}\Big).
\end{equation}
Substituting~\eqref{use} in~\eqref{ppx1} and~\eqref{ppx2}, we find that the only point of zero-gradient of the objective function of Problem~\eqref{workSet5}
in this case is given by:
\begin{align}
& \tilde x^4_1 = \frac{X_{12}x_2(1-z_2)+x_1(X_{22}z_2 -x^2_2)}{X_{22}-x^2_2}, \label{opthx1}\\
&\tilde x^4_2=\frac{x_1(x^2_2-X_{22}(1-z_1))-X_{12}x_2z_1}{x_1 x_2 - X_{12}}. \label{opthx2}
\end{align}
By feasibility we must have $g_2 > 0$, which by~\eqref{ppx2},~\eqref{opthx1} and using $X_{22} z_2 \geq x^2_2$ for all $(x,X,z) \in \bar \C$,
is satisfied if and only if:
$$
\Big(\frac{\tilde x^4_2}{z_1+z_2-1}-\frac{x_2-\tilde x^4_2}{1-z_1}\Big) h > 0.
$$
Moreover, by~\eqref{r2} we have
\begin{align*}
&\frac{\tilde x^4_2}{z_1+z_2-1}-\frac{x_2-\tilde x^4_2}{1-z_1} < 0 \quad \Leftrightarrow \quad \tilde x^4_2 < (z_1+z_2-1)\frac{x_2}{z_2}\\
\Leftrightarrow\quad & x_1 z_2(x^2_2-X_{22}(1-z_1))-X_{12} x_2 z_1 z_2 < x_2 (z_1+z_2-1)(x_1 x_2 - X_{12}) \\
\Leftrightarrow\quad & x_1 (X_{22}z_2-x^2_2)+(1-z_2)x_2 X_{12} > 0,
\end{align*}
where the validity of the last inequality follows since for all $(x,X,z) \in \bar \C$ we have $X_{22}z_2-x^2_2 \geq 0$. Hence we conclude
that $g_2 > 0$, if and only if $h < 0$; \ie
\begin{equation}\label{usethis}
\tilde x^4_1 \tilde x^4_2 > (z_1 + z_2 -1) X_{12}.
\end{equation}
Moreover, for every $(x,X,z)\in \U_2$, the above inequality can be equivalently written as:
\begin{align}\label{reg9}
x_1^2& (x_2^2 - X_{22} (1 - z_1)) (X_{22} z_2-x_2^2)
\nonumber\\
& > 2 x_1 x_2 X_{12} z_1 (X_{22} z_2-x_2^2) - X_{12}^2 (X_{22}(z_1+z_2-1) + x_2^2 (1 - 2 z_1 - z_2(1 -z_1))).
\end{align}
Now we examine the bounds constraints~\eqref{f2}.
First notice that since by~\eqref{opthx1} we have $\tilde x^4_1 \geq 0$, and inequality~\eqref{usethis} implies that at any point with $g_2 \geq 0$, we also
have $\tilde x^4_2 \geq 0$.
Moreover, it can be checked that for every $(x,X,z)\in \U_2$, the constraints $\tilde x^4_1 \leq x_1$ and $\tilde x^4_2 \leq x_2$ are satisfied if and only if $x_1 X_{22} \geq X_{12} x_2$. To see the validity of the latter inequality, suppose by contradiction that $X_{12} x_2 > x_1 X_{22}$. Then using the fact that $X_{22} z_2 \geq x^2_2$ for all $(x,X,z) \in \bar \C$, we get $X_{12} x_2 > x_1 X_{22} \geq x_1 \frac{x_2^2}{z_2}$, implying $X_{12} z_2 > x_1 x_2$, which is in contradiction with inequality~\eqref{r2} defining $\U_2$. Therefore, the point defined by~\eqref{opthx1} and~\eqref{opthx2} satisfies the bound constraints~\eqref{f2}.

Let us denote by $\Ri_7$ the region defined by inequalities~\eqref{r2} and~\eqref{reg9}. We then conclude that for all $(x,X,z) \in \Ri_7$
an optimal solution of Problem~\eqref{workSet5} is given by~\eqref{opthx1} and~\eqref{opthx2}. Thus, by Lemma~\ref{projeq} we get
\begin{align*}
\tilde \S \cap \Ri_7 =\Big\{(x,X,z): \; & X_{22} \geq \frac{x^2_2}{z_2}, \;(X_{11}-x^2_1) (X_{22}-x^2_2) \geq (X_{12}-x_1 x_2)^2,\; x_1, x_2 > 0, \nonumber\\
&  z_1, z_2 < 1\Big\} \cap \Ri_7,
\end{align*}
which in turn implies that
\begin{align}\label{peice6}
\cl(\tilde \S \cap \Ri_7) =\Big\{(x,X,z): &\; X_{11}\geq x^2_1, \; X_{22} \geq \frac{x^2_2}{z_2}, \;(X_{11}-x^2_1) (X_{22}-x^2_2) \geq (X_{12}-x_1 x_2)^2,\nonumber\\
&x_1, x_2 \geq 0, \; z_1, z_2 \leq 1\Big\} \cap \cl(\Ri_7),
\end{align}
where $\cl(\Ri_7)$ is obtained from $\Ri_7$ by replacing the strict inequalities by nonstrict inequalities.
\vspace{0.1in}

\emph{This concludes the proof of Part~(IV) of Theorem~\ref{theTh} for $i=7$.}

\end{itemize}
%{\color{red}
%Need to take the closure of $X_{22} \geq \frac{x^2_2}{z_2}, \;(X_{11}-x^2_1) (X_{22}-x^2_2) \geq (X_{12}-x_1 x_2)^2$, $x_1, x_2 > 0, X_{12} > 0$, $z_1, z_2 \in (0,1)$, intersected with
%inequalities~\eqref{r2} and~\eqref{reg9}.
%}

%To complete the convex hull characterization over Region~(2), we show that any point $(x,X,z) \in \conv(\S)$ satisfying~\eqref{r2} satisfies either condition~\eqref{reg8}
%or condition~\eqref{reg9}. This in turn implies that $\conv(\S)$ over Region~(2) is given by inequalities~\eqref{peice5} and~\eqref{peice6}.
%$$
% x^2_1 (X_{22} z_2 - x^2_2) < \frac{1}{(z_1+z_2-1)(1-z_1)}\Big(X_{12}z_1 z_2-x_1 x_2 (z_1+z_2-1)\Big)^2,
%$$
%\begin{align}
% &X_{12}^2 (X_{22}(z_1+z_2-1) + x_2^2 (1 - 2 z_1 - z_2(1 -z_1)))\\
% < & (2 x_2 X_{12} z_1 -x_1 (x_2^2 - X_{22} (1 - z_1))) x_1(X_{22} z_2-x_2^2)
%\end{align}

\paragraph{The boundary $g_2 =0$.}
By the above analysis, at any $(x,X,z)\in \U_2 \setminus (\Ri_6 \cup \Ri_7)$, the following
convex optimization problem has no point of zero-gradient:
\begin{align}
\label{workSet7}
\tag{$\bar {\rm P'}_{x,X,z}$}
\min_{\tilde x^4_1,\tilde x^4_2} \quad & \frac{(\tilde x^4_1)^2}{z_1+z_2-1}+\frac{(x_1-\tilde x^4_1)^2}{1-z_2}
+\frac{\Big(X_{12}-\frac{\tilde x^4_1 \tilde x^4_2}{z_1+z_2-1}\Big)^2}{X_{22}-\frac{(\tilde x^4_2)^2}{z_1+z_2-1}-\frac{(x_2 - \tilde x^4_2)^2}{1-z_1}},\\
\st \quad & X_{22}- \frac{(\tilde x^4_2)^2}{z_1+z_2-1}-\frac{(x_2 - \tilde x^4_2)^2}{1-z_1} \geq 0\label{ff1}.
\end{align}
Inequality~\eqref{ff1} can be equivalently written as:
\begin{equation}\label{o1}
\lambda_4 \frac{x_2}{z_2} - \frac{\sqrt{(X_{22}z_2-x^2_2)(1-z_1)\lambda_4}}{z_2} \leq \tilde x^4_2 \leq \lambda_4 \frac{x_2}{z_2} + \frac{\sqrt{(X_{22}z_2-x^2_2)(1-z_1)\lambda_4}}{z_2},
\end{equation}
where as before $\lambda_4 = z_1 + z_2 -1$.
Over Region~$\U_2$, all points of zero-gradient of the objective function of Problem~\eqref{workSet7} defined by~\eqref{opth0}, and~\eqref{opthx1}-\eqref{opthx2} satisfy $\tilde x^2_4 \leq \lambda_4 \frac{x_2}{z_2}$. To see this, first consider~\eqref{opth0}; \ie $\tilde x^4_2 = X_{12} \frac{z_1}{x_1}$. In this case $\tilde x^4_2 \leq \lambda_4 \frac{x_2}{z_2}$
follows directly from~\eqref{r2} defining $\U_2$. Next, consider~\eqref{opthx2}; in this case $\tilde x^4_2 \leq \lambda_4 \frac{x_2}{z_2}$, can be equivalently written as
$$(1-z_1) (X_{12} x_2 (1 - z_2) + x_1 (X_{22} z_2-x_2^2)) \geq 0,$$
 whose validity follows from the fact that $X_{22} z_2-x_2^2 \geq 0$ for all $(x,X,z) \in \bar \C$.

Therefore, in constraint~\eqref{o1} the upper bound on $\tilde x_2^4$ is redundant. Hence in Problem~\eqref{workSet7} we can replace constraint~\eqref{ff1}
by the following inequality:
\begin{equation}\label{o2}
\tilde x^4_2 \geq \lambda_4 \frac{x_2}{z_2} - \frac{\sqrt{(X_{22}z_2-x^2_2)(1-z_1)\lambda_4}}{z_2}.
\end{equation}
Since Problem~\eqref{workSet7} has no point of zero-gradient inside its domain, its minimizer is attained either at some point of non-differentiability or
at the boundary of its domain; \ie $g_2 = 0$; in fact all points satisfying $g_2 = 0$ are non-differentiability points. Since $z_1 + z_2 > 1$, the only remaining point
of differentiability is $z_2 = 1$, however, the case with $z_2 =1$ can be considered as a limiting point of the analysis in the previous section and yields to
characterizations~\eqref{peice5} and~\eqref{peice6} with $z_2 = 1$. Hence, it only remains to consider $g_2 = 0$, which in turn implies $h = 0$.
In this case, the minimizer of Problem~\eqref{workSet7} is given by:
\begin{align}
&\tilde x^4_1 = \frac{X_{12} z_2}{x_2 - \sqrt{\frac{(X_{22}z_2-x^2_2)(1-z_1)}{z_1+z_2-1}}} \label{loptx1}\\
&\tilde x^4_2 =  (z_1+z_2-1) \frac{x_2}{z_2} - \frac{\sqrt{(X_{22}z_2-x^2_2)(1-z_1)(z_1+z_2-1)}}{z_2}\label{loptx2}.
\end{align}
To prove that the above point is also a minimizer of Problem~\eqref{workSet5}, we must show that it satisfies constraints~\eqref{f2}.
It can be checked that $\tilde x^4_1 \geq 0$ and $\tilde x^4_2 \geq 0$, if
\begin{equation}\label{aux}
x^2_2 > X_{22} (1-z_1).
\end{equation}
We now show that~\eqref{aux} is valid since by assumption inequality~\eqref{reg8} is not satisfied; \ie
\begin{align}
&\Big(\frac{1-z_1}{z_1+z_2-1}\Big) x^2_1 (X_{22} z_2 - x^2_2) \leq \Big(x_1 x_2- X_{12} \frac{z_1 z_2}{z_1+z_2-1}\Big)^2, \label{hl}\\
\Leftrightarrow \quad & (X_{22}(1-z_1)-x_2^2)x^2_1 (z_1+z_2-1) \leq X_{12}z_1(X_{12}z_1 z_2-2x_1x_2(z_1+z_2-1)).\nonumber
\end{align}
The proof then follows since by~\eqref{r2} we have $X_{12}z_1 z_2-2x_1x_2(z_1+z_2-1) < 0$, and hence $X_{22}(1-z_1)-x_2^2 < 0$; \ie inequality~\eqref{aux} holds.
Hence we have $\tilde x^4_1 \geq 0$ and $\tilde x^4_2 \geq 0$.

Next, we prove $\tilde x^4_1 \leq x_1$; by~\eqref{r2} and~\eqref{loptx1}, this inequality can be equivalently written as:
\begin{equation}\label{aux2}
\Big(\frac{1-z_1}{z_1+z_2-1}\Big) x^2_1 (X_{22}z_2-x^2_2) \leq (x_1 x_2 - X_{12} z_2)^2.
\end{equation}
For every $(x,X,z)\in \U_2$ we have
$$\Big(x_1 x_2- X_{12} \frac{z_1 z_2}{z_1+z_2-1}\Big)^2 < (x_1 x_2 - X_{12} z_2)^2.$$
The validity of inequality~\eqref{aux2} then follows from~\eqref{hl}.
Finally, by~\eqref{loptx2} we have $\tilde x^4_2 \leq (z_1+z_2-1)\frac{x_2}{z_2} \leq x_2$. We conclude that for every $(x,X,z) \in \Ri_8$, where
%\begin{align}\label{reg10}
%\Ri_8:=&\Big\{(x,X,z): \; X_{12} z_1 z_2 < x_1 x_2 (z_1 + z_2 -1),\; X_{12} > 0, \; z_1, z_2 > 0\nonumber\\
%&(1-z_1)(z_1+z_2-1) x^2_1 (X_{22} z_2 - x^2_2) < \Big(X_{12}z_1 z_2-x_1 x_2 (z_1+z_2-1)\Big)^2,\\
%&x_1^2 (x_2^2 - X_{22} (1 - z_1)) (X_{22} z_2-x_2^2)
% \leq  2 x_1 x_2 X_{12} z_1 (X_{22} z_2-x_2^2)\nonumber\\
% & \qquad -X_{12}^2 (X_{22}(z_1+z_2-1) + x_2^2 (1 - 2 z_1 - z_2(1 -z_1)))\Big\},\nonumber
%\end{align}
\begin{equation}\label{reg10}
\Ri_8:=\U_2 \setminus (\Ri_6 \cup \Ri_7),
\end{equation}
an optimal solution of Problem~\eqref{workSet5}
is given by~\eqref{loptx1} and~\eqref{loptx2}. Therefore, by Lemma~\ref{projeq}, we get:
\begin{align*}
\tilde \S \cap \Ri_8= \Big\{&(x,X,z): \; z_1 (1-z_2)\Big(X_{11}-\frac{x^2_1}{z_1}\Big) x^2_2 \geq (z_1+z_2-1)\Big(X_{12}\frac{z_1 z_2}{W}-x_1 x_2\Big)^2, \; x_1, x_2 > 0, \nonumber\\
&  z_1, z_2 \leq 1\Big\}  \cap \Ri_8,
\end{align*}
where we define
$$
W := (z_1+ z_2-1)- \frac{1}{x_2}\sqrt{(X_{22}z_2-x^2_2)(1-z_1)(z_1+z_2-1)}.
$$
It can be checked that, by~\eqref{aux}, we have $W > 0$. Moreover, using a similar line of arguments as in previous cases, it can be shown that
\begin{align}\label{peice7}
\cl(\tilde \S \cap \Ri_8)=  \Big\{&(x,X,z): \; X_{11} \geq \frac{x^2_1}{z_1}, \; X_{22} \geq \frac{x^2_2}{z_2}, \; x_1, x_2 \geq 0, \; z_1, z_2 \leq 1\nonumber \\
&z_1 (1-z_2)\Big(X_{11}-\frac{x^2_1}{z_1}\Big) x^2_2 \geq (z_1+z_2-1)\Big(X_{12}\frac{z_1 z_2}{W}-x_1 x_2\Big)^2 \Big\} \cap \cl(\Ri_8),
\end{align}
where $\cl(\Ri_8)$ is obtained by replacing strict inequalities in~\eqref{reg10} by non-strict inequalities.
\vspace{0.1in}

\emph{This concludes the proof of Part~(V) of Theorem~\ref{theTh} for $i=8$.}
\vspace{0.1in}

\begin{comment}
\section{Separation over the convex hull}
\label{separate}

Define
$$
\C=\Big\{(x,X,z): \; X_{11} \geq \frac{x^2_1}{z_1}, \; X_{22} \geq \frac{x^2_2}{z_2}, \; X_{12} \geq 0, \; x_1, x_2 \geq 0, \; z_1, z_2 \in [0,1]\Big\}.
$$

We start by defining the separation problem:

\paragraph{The separation problem.} Given $(\tilde x, \tilde X, \tilde z) \in \C$, decide whether $(\tilde x, \tilde X, \tilde z) \in \overline{\conv}(\S)$
and if not, return a valid inequality for $\overline{\conv}(\S)$ that separates $(\tilde x, \tilde X, \tilde z)$ from $\overline{\conv}(\S)$.

Since $x \in \C$, we can discard $\tilde \S \cap \Ri_1$. Moreover, by Corollary~\eqref{minisep}, we can also separate over $\tilde \S \cap \Ri_4$ and $\tilde \S \cap \Ri_5$.
Hence, in the following, we consider
\end{comment}

We conclude the proof of Theorem~\ref{theTh} by remarking that it is simple to check that $\Ri_i \cap \Ri_j = \emptyset$ for all $i \neq j \in \{1,\ldots,8\}$
and $\bigcup_{i=1}^8{\Ri_i} \supset \overline {\conv}(\S_2)$. In fact, from the proof it follows that $\bigcup_{i=1}^8{\Ri_i} \supset \bar \C$, where $\bar \C$ is defined by~\eqref{simple}; we use this property in the development of our separation algorithm.

\section{Technical Lemmata}
\label{lemmata}
In this section, we state and prove all technical lemmata that we used to prove Theorem~\ref{theTh}, Proposition~\ref{usefull}, and Proposition~\ref{sepProp}. In addition, we provide a brief motivation for each of them.

\vspace{0.1in}

The first lemma enables us to obtain the closure of the convex hull of the disjunctive set $\S_2 = \cup_{i=1}^k{\P_i}$ in a sequential manner.
%that is, by first obtaining the closure of the convex hull for each of its disjunctions $\P_i$, $i \in \{1,\ldots,k\}$, followed by taking
%the closure of the convex hull of the union of the resulting closed convex sets $\overline{\conv}(\P_i)$.

\begin{lemma}\label{lemClose}
Let $k \geq 2$ and consider the sets $\C, \P_1, \ldots, \P_k$ in $\R^n$ such that $\C = \bigcup_{i=1}^k{\P_i}$. Then
$$
\overline{\conv}(\C) = \overline{\conv}\Big(\bigcup_{i=1}^k{\overline{\conv}{(\P_i)}}\Big).
$$
\end{lemma}

\begin{proof}
Since $\conv(\C) \subseteq \conv(\overline{\conv}(\P_1) \cup \ldots \cup \overline{\conv}(\P_k))$, we have
$\overline{\conv}(\C) \subseteq \overline{\conv}(\overline{\conv}(\P_1) \cup \ldots \cup \overline{\conv}(\P_k))$. Hence it suffices to prove the reverse inclusion.

We observe that $\bigcup_{i=1}^k{\overline{\conv}{(\P_i)}} \subseteq \overline{\conv}(\C)$ and, since $\overline{\conv}(\C)$ is closed and convex, we deduce that $\overline{\conv}(\bigcup_{i=1}^k{\overline{\conv}{(\P_i)}}) \subseteq \overline{\conv}(\C)$.
%, it follows that
%$\conv(\overline{\conv}(\P_1) \cup \ldots \cup \overline{\conv}(\P_k)) \subseteq \overline{\conv}(\C)$. Moreover, since $\overline{\conv}(\C)$ is closed, we get
%$\overline{\conv}(\overline{\conv}(\P_1) \cup \ldots \cup \overline{\conv}(\P_k)) \subseteq \overline{\conv}(\C)$, and this completes the proof.
\end{proof}

It is well-known that a matrix is PSD if and only if all of its principal minors
are nonnegative. The next lemma indicates that for $3\times 3$ matrices, the nonnegativity of only a subset of principal minors is equivalent
to positive semidefiniteness. This result simplifies our convex hull characterization.

\begin{lemma}\label{lem1}
The constraint
\begin{equation}\label{psd1}
\begin{pmatrix}
1   & x_1    & x_2 \\
x_1 & X_{11} & X_{12} \\
x_2 & X_{12} & X_{22}
\end{pmatrix} \succeq 0,
\end{equation}
is equivalent to the following inequalities:
\begin{equation}\label{psd2}
X_{11} \geq x^2_1, \quad X_{22} \geq x^2_2, \quad (X_{11}-x^2_1)(X_{22}-x^2_2) \geq (X_{12}-x_1 x_2)^2.
\end{equation}
\end{lemma}
\begin{proof}
By characterization of PSD matrices in terms of the sign of their principal minors, constraint~\eqref{psd1} can be equivalently written as system \eqref{e1}-\eqref{e3} below:
\begin{align}
& X_{11} \geq x^2_1, \quad X_{22} \geq x^2_2 \label{e1}\\
& X_{11} X_{22} \geq X^2_{12} \label{e2}\\
& (X_{11}-x^2_1)(X_{22}-x^2_2) \geq (X_{12}-x_1 x_2)^2 \label{e3}
\end{align}
It then suffices to show that inequalities~\eqref{e1} and~\eqref{e3} imply inequality~\eqref{e2}. Consider a point $(x_1, x_2, X_{11},X_{12},X_{22})$ that satisfies inequalities~\eqref{e1} and~\eqref{e3} and assume first that $X_{12}\geq 0$. We split the proof in two cases:
\begin{itemize}[leftmargin=*]
\item $X_{12} \leq x_1 x_2$: since by~\eqref{e1} we have $X_{11}\geq 0$ and $X_{22}\geq 0$, it follows that
    \begin{equation*}
    \sqrt{X_{11}X_{22}}-X_{12} \geq x_1x_2-X_{12} \geq 0,
    \end{equation*}
    where the first inequality follows from~\eqref{e1} and the second inequality follows from the assumption. Since $X_{12}\geq 0$, we deduce that inequality~\eqref{e2} is satisfied.

\item $X_{12} \geq x_1 x_2$: since by~\eqref{e1} we have $X_{11} \geq 0$ and $X_{22} \geq 0$, it follows that~\eqref{e3} can be equivalently written as:
\begin{align*}
\Big(\sqrt{X_{11}X_{22}}-X_{12}\Big)\Big(\sqrt{X_{11}X_{22}}+X_{12}-2x_1 x_2\Big) \geq \Big(x_1 \sqrt{X_{22}}-x_2 \sqrt{X_{11}}\Big)^2
\end{align*}
Moreover, by~\eqref{e1} we have  $\sqrt{X_{11}X_{22}}+X_{12}-2x_1 x_2 \geq X_{12}-x_1 x_2 \geq 0$, where the second inequality follows from assumption.
Hence, we must have $\sqrt{X_{11}X_{22}}-X_{12} \geq 0$. Since $X_{12}\geq 0$, we deduce that inequality~\eqref{e2} is satisfied.
\end{itemize}

To conclude, if $X_{12}< 0$, we consider the point
$$(\tilde x_1, \tilde x_2,\tilde X_{11}, \tilde X_{12}, \tilde X_{22}):= (-x_1, x_2, X_{11},-X_{12},X_{22})$$
which also satisfies inequalities~\eqref{e1} and~\eqref{e3}. Since $\tilde X_{12}> 0$, by the argument above the point $(\tilde x_1, \tilde x_2,\tilde X_{11}, \tilde X_{12}, \tilde X_{22})$ satisfies inequality \eqref{e2}. Hence also the point $(x_1, x_2, X_{11},X_{12},X_{22})$ satisfies \eqref{e2}, as desired.
\end{proof}

Notice that the result of Lemma~\ref{lem1} can be restated for any $3\times3$ matrix $\A$ after dividing all entries by the first entry $a_{11}$. That is,
one needs to check the positivity of $a_{11}$ together with inequalities~\eqref{psd2} after a proper rescaling, implying it is enough to check the sign of four out of seven principal minors of $\A$. Notice that assuming $a_{11} > 0$
is without loss of generality, since if $a_{11} = 0$, all entries in the first row and the first column of $\A$ must be zero, in which case the PSD check should be done for a $2\times2$ matrix.

The next result enables us to simplify the description of $\Sigma$ defined by~\eqref{sbar} by discarding some parts of its boundary.

\begin{lemma}[Theorem~6.9 in~\cite{rf70}]\label{drs}
Let $\C_1,\ldots, \C_m\subseteq \R^n$ be nonempty convex sets, and let $\C_0 := \conv(\C_1 \cup \ldots \cup \C_m)$. Then
$$
{\rm ri}(\C_0) = \bigcup\Big\{\lambda_1 {\rm ri}(\C_1)+ \ldots \lambda_m {\rm ri}(\C_m): \; \lambda_i > 0, \forall i \in [m], \; \sum_{i=1}^m{\lambda_i} =1 \Big\}.
$$
\end{lemma}

To construct the convex hull $\S_2$ in the original space, we define sets $\Ri_i$, $i \in \{1,\ldots,8\}$ satisfying
$\cup_{i=1}^8{\Ri_i} \supseteq \overline{\conv}(\S_2)$. Subsequently, we characterize the intersection of the convex hull of $\S_2$ with each $\Ri_i$ in the original space.
Finally, using the following lemma, we obtain $\overline{\conv}(\S_2)$ from these piece-wise descriptions.

\begin{lemma}\label{intcl}
Consider a set $\C$ and $k$ sets $\Ri_1,\dots, \Ri_k$ such that $\cup_{i=1}^k{\Ri_i} \supseteq \C$. Then
$$
\cl(\C)=\bigcup_{i=1}^k\cl(\C \cap \Ri_i).
$$
\end{lemma}
\begin{proof}
We have
$$\cl(\C)=\cl\Big(\bigcup_{i=1}^k(\C\cap \Ri_i)\Big)=\bigcup_{i=1}^k\cl(\C\cap \Ri_i),$$
where in the last equality we used that closure commutes with a finite union.
%
%inclusion $\bigcup_{i=1}^k\cl(\tilde \S\cap \Ri_i)\subset \cl(\bigcup_{i=1}^k(\tilde \S\cap \Ri_i))$ is trivial, while if
%$x\in \cl(\bigcup_{i=1}^k(\tilde \S\cap \Ri_i))$, then there exists a sequence $\{x_n\}\subset \bigcup_{i=1}^k(\tilde \S\cap \Ri_i)$ such that $x_n\to x$. Hence, up to passing to a subsequence, there exists $j\in [k]$ such that $\{x_n\}\subset \tilde \S\cap \Ri_j$, implying $x\in \cl(\tilde \S \cap \Ri_i)\subset \bigcup_{i=1}^k\cl(\tilde \S\cap \Ri_i)$.
\end{proof}

The next lemma provides us with a key tool to perform the projection. It essentially says that the projection operation is equivalent to solving a parametric optimization problem.
This result is of independent interest as it provides us with a general technique to obtain projections of sets defined by nonlinear inequalities.

\begin{lemma}\label{projeq}
 Consider a function $f:(x,y)\in \R^n\times \R^m\to \R\cup \{\infty\}$ and two sets $\G\subset \R^n$, $\D\subset \R^n\times \R^m$. Define
 $$\Q:=\{x\in \G\,: \,\exists y\in  \R^m \mbox{ s.t. } f(x,y)\leq 0, \; (x,y)\in \D\}.$$
 For every $x\in \G$, consider the following optimization problem:
 \begin{align}\label{projOpt}
 \tag{AUX}
 \min \quad & f(x,y)\\
 \st  \quad & (x,y) \in \D \nonumber
 \end{align}
 and consider a set
 $$
 \E\subseteq \{x\in \G\, :\, \mbox{Problem \eqref{projOpt} admits a minimizer, denoted by $y_x$}\}.
 $$
 Then
 $$\Q\cap \E=\Q_{\E}:=\{x\in \E\, :\, f(x,y_x)\leq 0, \; (x,y_x)\in \D\}.$$
 \end{lemma}
 \begin{proof}
 For every $x\in \Q\cap \E$ there exists $y\in  \R^m$ such that $f(x,y)\leq 0, (x,y)\in \D$. Hence, by definition of $y_x$, it follows that $f(x,y_x)\leq f(x,y)\leq 0$ and $(x,y_x)\in \D$. We deduce that $x\in \Q_{\E}$. Analogously, for every $x\in \Q_{\E}$ it holds $f(x,y_x)\leq 0$, and $(x,y_x)\in \D$. Hence $x\in \Q\cap \E$.
 \end{proof}

Consider now a convex function $f$ defined over a convex set $\D$ and consider a box $\H$.
It is well understood that if the minimum $f$ over $\D$, denoted by $z$ is attained outside $\H$, the minimum of $f$ over $\D \cap \H$ denoted by $\tilde z$
is attained over a face of $\H$.
The following lemma shows that, depending on the relative location of $z$ with respect to $\H$,  $\tilde z$ can only be attained over certain faces of $\H$.
This drastically simplifies our search for optimal solutions of Problem~\eqref{workSet3}.

\begin{lemma}\label{convbnd}
Consider a convex set $\D\subset \R^n$, and the box
  $$\H:=[l_1,u_1]\times \dots \times [l_n,u_n], \quad \mbox{where} \quad l_i,u_i\in \mathbb R\cup\{\pm \infty\}, \; \forall i=1,\dots,n.$$
   Consider a convex function $f:\D\to \R\cup\{\infty\}$ that attains its minimum value at a point $z\in \D\setminus \H$. Assume there exists $i<j\in [n]$ such that $u_k\geq l_k>z_k$ for every $k=1,\dots,i$, $z_k>u_k \geq l_k $ for every $k=i+1,\dots,j$ and $u_k\geq z_k\geq l_k$ for every $k=j+1,\dots, n$.
   Then
 \begin{equation}\label{eq}
 \min_{\H\cap \D} f= \min_{\I\cap \D} f,
 \end{equation}
   where
   $$\I:=\Big(\bigcup_{k=1}^i  [l_1,u_1]\times \dots \times \{l_k\} \times \dots \times [l_n,u_n] \Big)\cup \Big(\bigcup_{k=i+1}^j  [l_1,u_1]\times \dots \times \{u_k\} \times \dots \times [l_n,u_n]\Big)\subset \partial \H.$$
 \end{lemma}
 \begin{proof}
 Since $\H$ and $\D$ are convex, for every $x\in \H\cap \D$ the line segment $\{(1-s)z+sx:s\in [0,1]\}\subset \D$. Moreover, since $z\in \D\setminus \H$, for every $x\in  \H\cap \D$ there exists $s_x\in [0,1]$ such that $(1-s_x)z+s_xx\in \I\cap \D$. Using that $f$ is convex and $z$ is a minimizer, we conclude that
 $$f((1-s_x)z+s_xx)\leq (1-s_x)f(z)+s_xf(x)\leq (1-s_x)f(x)+s_xf(x)= f(x).$$
 As $\I\subset \H$, equality \eqref{eq} holds.
 \end{proof}

The next lemma provides us with a key tool to generate supporting hyperplanes of $\overline{\conv}(\S_2)$ during
the course of our separation algorithm. Roughly speaking, given our piece-wise characterization of the form $\cl(\conv(\S_2) \cap \Ri_i)$ for $i \in \{1,\ldots,8\}$,
this result provides a local characterization of the of the boundary of $\overline{\conv}(\S_2)$.

\begin{lemma}\label{bps}
Consider an open set $\mathcal A\subset \mathbb{R}^n$, four continuously differentiable functions $f_1,f_2,h_1,h_2:\mathcal A \to \R$, and two sets $\Ri, \C\subseteq \mathcal A$ such that $\C$ is an $n$-dimensional convex set and
$$
\Ri=\{x: f_1(x)>0,f_2(x)\geq 0\}, \quad \C\cap \Ri=\{x: h_1(x)\geq 0,h_2(x)\geq 0, f_1(x)>0,f_2(x)\geq 0\}.
$$
Assume that
\begin{equation}\label{ass}
\mbox{$\forall x\in \C\cap \Ri$ such that $h_1(x)=0$ and $\forall s>0$ it holds $\B_s(x)\setminus (\C\cap \Ri)\neq \emptyset$,}
\end{equation}
and that there exists $\hat x \in \Ri$ such that $h_1(\hat x)=0$, $\nabla h_1(\hat x)\neq 0$, $h_2(\hat x)>0$, $f_1(\hat x)>0$ and $f_2(\hat x)>0$.
Then there exists $r>0$ such that
\begin{equation}\label{00}
\{x: h_1(x)=0\}\cap \B_r(\hat x)= \partial \C \cap \B_r(\hat x).
\end{equation}
\end{lemma}
\begin{proof}
We first claim that
\begin{equation}\label{11}
\{x: h_1(x)=0, h_2(x)\geq 0, f_1(x)> 0,f_2(x)> 0\}\subset \partial \C.
\end{equation}
By the definition of $\C\cap \Ri$ and assumption \eqref{ass}, it follows that for every $x\in \{x: h_1(x)=0,h_2(x)\geq 0, f_1(x)> 0,f_2(x)> 0\}$ and for every $s>0$, we have that $\B_s(x)\cap (\C\cap \Ri)\neq \emptyset$ and $\B_s(x)\setminus (\C\cap \Ri)\neq \emptyset$. Hence, we conclude that
$$\{x: h_1(x)=0,h_2(x)\geq 0, f_1(x)> 0,f_2(x)> 0\}\subset \partial (\C\cap \Ri).$$

Since $f_1(x), f_2(x)$ are continuous functions, the set $\{x: f_1(x)>0,f_2(x)>0\}\subseteq \Ri$ is open, hence $\{x: f_1(x)>0,f_2(x)>0\}\cap \partial \Ri=\emptyset$. We deduce that
$$\{x: h_1(x)=0,h_2(x)\geq 0, f_1(x)> 0,f_2(x)> 0\}\subset \partial (\C\cap \Ri) \setminus \partial \Ri.$$
To conclude the proof of \eqref{11}, we are left to prove that
$$\partial (\C\cap \Ri) \setminus \partial \Ri  \subset \partial \C.$$
 The last inclusion can be proved as follows. Fix $x\in \partial (\C\cap \Ri)$, by definition for every $n\in \mathbb N$ there exists $y_n\in \B_{1/n}(x)\cap (\C\cap \Ri)$ and $z_n\in \B_{1/n}(x)\setminus (\C\cap \Ri)=(\B_{1/n}(x)\setminus \C)\cup (\B_{1/n}(x)\setminus \Ri)$. Hence, up to choosing a subsequence $n_k$, either $z_{n_k}\in \B_{1/n_k}(x)\setminus \C$ for every $k\in \mathbb N$, or $z_{n_k}\in \B_{1/n_k}(x)\setminus \Ri$ for every $k\in \mathbb N$. We deduce that $x\in \partial \C\cup \partial \Ri$. Hence $\partial (\C\cap \Ri)\subset  \partial \C\cup \partial \Ri$ and in particular $\partial (\C\cap \Ri) \setminus \partial \Ri  \subset \partial \C$, as desired.

Since $h_1(\hat x)=0$, $h_2(\hat x)>0$, $f_1(\hat x)>0$, $f_2(\hat x)>0$, by the continuity of $h_2,f_1,f_2$ at $\hat x$ we deduce that there exists $r>0$ such that
 \begin{equation}\label{22}
\{x: h_1(x)=0\}\cap \B_r(\hat x)= \{x: h_1(x)=0, h_2(x)> 0, f_1(x)> 0,f_2(x)> 0\} \cap \B_r(\hat x)\subset \partial \C \cap \B_r(\hat x),
\end{equation}
where the last inclusion follows from \eqref{11}.
Since $h_1$ is continuously differentiable and $\nabla h_1(\hat x)\neq 0$, by the implicit function theorem, up to consider a smaller $r>0$, we deduce that $\{x: h_1(x)=0\}\cap \B_r(\hat x)$ is a continuously differentiable $(n-1)$-dimensional surface. Since $\C$ is an $n$-dimensional convex set, we also know that $ \partial \C \cap \B_r(\hat x)$ is a Lipschitz $(n-1)$-dimensional surface. This, combined with \eqref{22}, implies \eqref{00}.

\end{proof}

\begin{footnotesize}
\bibliographystyle{plain}
\bibliography{biblio}
\end{footnotesize}

\end{document}